\documentclass[11pt]{article}
\usepackage{palatino}
\usepackage{amsmath}
\usepackage{graphicx,psfrag,epsf}
\usepackage{enumerate}
\usepackage{amsfonts}
\usepackage{amsthm}
\usepackage{booktabs}
\usepackage{float}
\usepackage{color}
\usepackage{caption}
\usepackage{verbatim}
\usepackage{amssymb}
\usepackage{algorithm}
\usepackage{algpseudocode}
\newcommand{\blind}{1}
\usepackage{subcaption}
\usepackage{setspace}
\usepackage[normalem]{ulem}
\usepackage{lscape}
\usepackage{makecell} 

\definecolor{darkblue}{rgb}{0.0,0.0,0.55}
\RequirePackage[colorlinks,citecolor=darkblue,urlcolor=darkblue,linkcolor=darkblue]{hyperref} 
\usepackage{xr} 

\usepackage{algpseudocode}
\errorcontextlines\maxdimen
\makeatletter 
\newcommand*{\algrule}[1][\algorithmicindent]{\makebox[#1][l]{\hspace*{.5em}\vrule height .75\baselineskip depth .25\baselineskip}}%

\makeatletter
\g@addto@macro\normalsize{%
  \setlength\abovedisplayskip{9pt}
  \setlength\belowdisplayskip{9pt}
  \setlength\abovedisplayshortskip{5pt}
  \setlength\belowdisplayshortskip{5pt}
}

\newcount\ALG@printindent@tempcnta
\def\ALG@printindent{%
    \ifnum \theALG@nested>0
        \ifx\ALG@text\ALG@x@notext
            \addvspace{-3pt}
        \else
            \unskip
            \ALG@printindent@tempcnta=1
            \loop
                \algrule[\csname ALG@ind@\the\ALG@printindent@tempcnta\endcsname]%
                \advance \ALG@printindent@tempcnta 1
            \ifnum \ALG@printindent@tempcnta<\numexpr\theALG@nested+1\relax
            \repeat
        \fi
    \fi
    }%
\usepackage{etoolbox}
\patchcmd{\ALG@doentity}{\noindent\hskip\ALG@tlm}{\ALG@printindent}{}{\errmessage{failed to patch}}
\makeatother

\newcommand{\ben}{\begin{enumerate}}
\newcommand{\een}{\end{enumerate}}
\newcommand{\beq}{\begin{equation}}
\newcommand{\eeq}{\end{equation}}
\newcommand{\bde}{\begin{description}}
\newcommand{\ede}{\end{description}}

\newcommand{\argmax}{\operatornamewithlimits{argmax}}

\newcommand{\NormRV}{\mathcal{N}}

\usepackage[margin=1.00in, centering]{geometry}

\newtheorem{theorem}{Theorem}[section]
\newtheorem{corollary}[theorem]{Corollary}
\newtheorem{lemma}[theorem]{Lemma}
\newtheorem{assumption}[theorem]{Assumption}
\usepackage{bm}

\newcommand{\thetavec}{\ensuremath{{\bm{\theta}}}}

\newcommand{\Zmat}{\ensuremath{\bm{Z}}}

\newcommand{\Lambdamat}{\ensuremath{\bm{\Lambda}}}

\newcommand{\Omegamat}{\ensuremath{\bm{\Omega}}}
\newcommand{\Ommat}{\Omegamat}

\newcommand{\zerovec}{\ensuremath{{\bf 0}}}

\newcommand{\eps}{\ensuremath{\epsilon}}

\newcommand{\be}{\begin{equation}}
\newcommand{\ee}{\end{equation}}
\newcommand{\beqa}{\begin{eqnarray*}}
\newcommand{\eeqa}{\end{eqnarray*}}
\newcommand{\beqn}{\begin{eqnarray}}
\newcommand{\eeqn}{\end{eqnarray}}
\newcommand{\ba}{\begin{array}}
\newcommand{\ea}{\end{array}}
\newcommand{\bc}{\begin{center}}
\newcommand{\ec}{\end{center}}
\newcommand{\btab}{\begin{tabular}}
\newcommand{\etab}{\end{tabular}}

\newcommand{\mb}{\makebox}

\newcommand{\st}{\stackrel}

\newcommand{\Ind}{1\!\mathrm{l}}

\newcommand{\ind}{\, \raise-2pt\hbox{$\st{\mb{\scriptsize ind}}{\sim}$}\, }
\newcommand{\iid}{\, \raise-2pt\hbox{$\st{\mb{\scriptsize iid}}{\sim}$}\,}

\newcommand{\norm}[1]{\|#1\|}

\renewcommand{\P}{\mathrm{P}}

\newcommand{\Var}{\mathop{\rm Var}\nolimits}

\newcommand{\tr}{\mathrm{tr}}

\newcommand{\bX}{\bm{X}}
\newcommand{\bx}{\bm{x}}

\newcommand{\bbeta}{\bm{\beta}}

\newcommand{\bDelta}{\bm{\Delta}}
\newcommand{\bSigma}{\bm{\Sigma}}
\newcommand{\bOmega}{\bm{\Omega}}

\newcommand{\bmS}{\bm{S}}

\newcommand{\bmA}{\bm{A}}

\newcommand{\bmI}{\bm{I}}

\newcommand{\RR}{{\mathbb R}}

\mathchardef\given="626A
\mathcode`:="603A

\long\def\beginskip#1\endskip{}
\def\endskip{}

\def\arg{\mathop{\rm arg}}

\newcommand{\eig}{\mathrm{eig}}

\usepackage[sort,comma]{natbib}
\setcitestyle{citesep={;}}

\allowdisplaybreaks
\begin{document}

\def\spacingset#1{\renewcommand{\baselinestretch}%
{#1}\small\normalsize} \spacingset{1}


\if1\blind
{
  \title{\bf Precision Matrix Estimation under the Horseshoe-like Prior--Penalty Dual}
  \author{Ksheera Sagar\\
    Department of Statistics, Purdue University\\
    and \\
    Sayantan Banerjee\\
    Indian Institute of Management, Indore\\
    and\\
    Jyotishka Datta\\
    Department of Statistics, Virginia Polytechnic Institute and State University\\
    and\\
    Anindya Bhadra \\
    Department of Statistics, Purdue University 
    }
    \date{}
  \maketitle
} \fi

\if0\blind
{
  \bigskip
  \bigskip
  \bigskip
  \begin{center}
     {\LARGE{\bf Precision Matrix Estimation under the Horseshoe-like Prior--Penalty Dual \par}}
\end{center}
  \medskip
} \fi

\bigskip
\begin{abstract}
Precision matrix estimation in a multivariate Gaussian model is fundamental to network estimation. Although there exist both Bayesian and frequentist approaches to this, it is difficult to obtain good Bayesian and frequentist properties under the same prior--penalty dual. To bridge this gap, our contribution is a novel prior--penalty dual that closely approximates the graphical horseshoe prior and penalty, and performs well in both Bayesian and frequentist senses. A chief difficulty with the horseshoe prior is a lack of closed form expression of the density function, which we overcome in this article. In terms of theory, we establish posterior convergence rate of the precision matrix that matches the oracle rate, in addition to the frequentist consistency of the MAP estimator. In addition, our results also provide theoretical justifications for previously developed approaches that have been unexplored so far, e.g. for the graphical horseshoe prior. Computationally efficient EM and MCMC algorithms are developed respectively for the penalized likelihood and fully Bayesian estimation problems. In numerical experiments, the horseshoe-based approaches echo their superior theoretical properties by comprehensively outperforming the competing methods. A protein--protein interaction network estimation in B-cell lymphoma is considered to validate the proposed methodology.
\end{abstract}
\noindent%
{\it Keywords:}  graphical models; non-convex optimization; posterior concentration; posterior consistency; sparsity. 
\newpage
\spacingset{1.45}
	\section{Introduction}
	\label{sec:intro}
	High-dimensional precision matrix estimation under a multivariate normal model is a fundamental building block for network estimation, and a common thread connecting disparate applications such as inference on gene regulatory networks  \citep{sanguinetti2019gene}, econometrics \citep{fan2016overview, callot2019nodewise}, and neuroscience \citep{ryali2012estimation}. The frequentist solution to this problem is now relatively well understood and several useful algorithms exist; see \citet{pourahmadi2011covariance} for a detailed review. However, interested readers will quickly discern that the Bayesian literature on this problem is still sparse, barring some notable exceptions described in Section~\ref{sec:Related_work}.  The reason for this is simple: the focus of a Bayesian is on the entire posterior and quantification of uncertainty using the said posterior; a problem fundamentally more demanding computationally. Consequently, the virtues of probabilistic uncertainty quantification notwithstanding, the Bayesian treatment to precision matrix estimation has received relatively scant attention from impatient practitioners. Furthermore, a penalized likelihood estimate with good frequentist properties need not correspond to good Bayesian posterior concentration properties under the corresponding prior. A notable example of this in linear regression models is the lasso penalty \citep{tibshirani1996lasso}, and its Bayesian counterpart using the double exponential prior \citep{park_bayesian_2008}, for which \citet{castillo2015bayesian} assert: \emph{``the LASSO is essentially non-Bayesian, in the sense that the corresponding full posterior distribution is a useless object.''}  We address this gap in the literature in the context of graphical models. Our contribution is a novel prior--penalty dual that makes both fully Bayesian and fast penalized likelihood estimation feasible. The key distinguishing feature of our work is that we provide theoretical and empirical support for \emph{both} Bayesian and frequentist solutions to the problem under the same prior--penalty dual. It is shown that the Bayesian posterior as a whole concentrates around the truth and the penalized likelihood point estimate is consistent. To our knowledge, ours is the first work to establish these results using continuous shrinkage priors under an \emph{arbitrary sparsity pattern} in the true precision matrix. This is at a contrast to the current state of the art in theory that imposes additional constraints on the graph, e.g., banded-ness or more general decomposable structures \citep{banerjee2014posterior, xiang2015high,liu2019empirical, lee2021estimating}. Typically these assumptions are made for computational and theoretical tractability rather than any intrinsic subject matter knowledge. The reason our work is able to avoid these restrictive assumptions is because we work with continuous ``global-local'' shrinkage priors and impose sparsity in a weak sense \citep[see, e.g., ][]{bhadra2019lasso}. 
	
	The motivating data set arises from a biological application. Protein--protein interaction networks have been found to play a crucial role in cancer \citep{ha2018personalized}. 
	One such significant effort in this direction is ``The Cancer Genome Atlas'' program \citep{weinstein2013cancer} that has collected data from over 7,700 patients across 32 different tumor types. From this repository, we retrieve proteomic data of 33 patients with ``Lymphoid Neoplasm Diffuse Large B-cell Lymphoma,'' which is a cancer that starts in white blood cells and spreads to lymph nodes. Our findings are contrasted with that of \citet{ha2018personalized}.
	\subsection{The Current State of the Art and Our Contributions in Context}
	\label{sec:Related_work}
	A Gaussian graphical model (GGM) remains popular as a fundamental building block for network estimation because of the ease of interpretation of the resulting precision matrix estimate: an inferred off-diagonal zero corresponds to conditional independence of the two corresponding nodes given the rest \citep[see, e.g.,][]{lauritzen1996}.
	
	Among the most popular frequentist approaches for estimating GGMs are the graphical lasso \citep{friedman2008glasso} and the graphical SCAD \citep{fan2009network}, which are respectively the graphical extensions of the lasso \citep{tibshirani1996lasso} and SCAD \citep{fan2001scad} penalties in linear models. Similarly, the CLIME estimator of \citet{cai2011constrained} is a graphical application of the Dantzig penalty \citep{candes2010power}.  \citet{fan2016overview} propose factor-based models for estimation of precision matrices, which are particularly attractive in financial applications where the precision matrix of outcome variables conditioned on some common factors being sparse is a sensible assumption. Alternatively, \citet{callot2019nodewise} opt for a node-wise regression approach using $\ell_1$ penalty for minimizing the risk of a Markowitz portfolio. The positive definiteness of their estimate is guaranteed asymptotically, which nevertheless remains hard to establish in finite samples; a common issue with node-wise regression approaches. \citet{zhang2014sparse} propose a new empirical loss termed the \emph{$D$-trace loss} to avoid computing the $\log$ determinant term in the $\ell_1$ penalized loss. Under certain conditions they also prove that the resulting estimate is identical to the CLIME estimate \citep{cai2011constrained}. A ridge type estimate for precision matrix termed ROPE is proposed by \citet{kuismin2017precision}, who use the squared Frobenius norm of the precision matrix as a penalty function. Another distribution free version of the ridge estimate is proposed by \citet{wang2015shrinkage}. An elastic net penalty  \citep{zou2005regularization} is used to determine the functional connectivity among brain regions by \citet{ryali2012estimation}. A comprehensive theoretical treatment for the rate of convergence of precision matrix estimates is given by \citet{lam2009sparsistency}. 
	
	The frequentist approaches listed above generally enjoy faster and more scalable computation, owing to being point estimates. Nevertheless, from a Bayesian perspective, a common theme with these penalized approaches is that the posterior concentration properties of the corresponding priors remain completely unexplored. Moving now to Bayesian methodologies for \emph{unstructured precision matrices}, the literature is relatively scant. \citet{wang2012bayesian} proposes a Bayesian version of the graphical lasso and uses a clever decomposition of the precision matrix to facilitate block Gibbs sampling and to guarantee the positive definiteness of the resulting estimate. \citet{banerjee2015bayesian} consider a similar prior structure as the Bayesian graphical lasso, with the exception that they put a large point-mass at zero for the off-diagonal elements of the precision matrix. Under assumptions of sparsity, they derive posterior convergence rates in the Frobenius norm, and also provide a Laplace approximation method for computing marginal posterior probabilities of models. Spike-and-slab variants with double exponential priors is proposed by \citet{gan2019bayesian}. A common issue with the spike-and-slab approach is the presence of binary indicator variables, which typically hinder posterior exploration and the Bayesian lasso estimate is known to be biased for large signals. Both of these issues are addressed by the graphical horseshoe estimate proposed by \citet{li2019graphical}, which is an application of the popular global-local horseshoe prior \citep{carvalho2010horseshoe} in GGMs. \citet{li2019graphical} provide considerable empirical evidence of superior performance over several competing Bayesian and frequentist approaches. Nevertheless, their theoretical results are limited to upper bounds on some Kullback--Leibler risk properties and the bias of the resulting estimate. Consequently, whether the graphical horseshoe posterior has correct concentration properties has remained an open question. Similarly, its frequentist dual: the penalized likelihood estimate, also remains unavailable, mainly because there is no closed form of the horseshoe prior or penalty; only a normal scale mixture representation. Both of these issues are resolved in the present paper. We propose a novel prior--penalty dual that closely approximates the graphical horseshoe prior with the density being available explicitly as well as a normal scale mixture, which has important implications in theory and in practice, and in both Bayesian and frequentist settings. Moreover, as a corollary to one of our main results, the posterior concentration properties of the graphical horseshoe is also established, for the first time.
	
\section{Formulation of the Prior--Penalty Dual}
	\label{prior_specification}
	We begin the formal treatment by pointing the interested readers to Supplementary Section~\ref{sec:notations} for a summary of mathematical notations used in the paper. Let $\bX^{(n)} = (\bX_1,\ldots,\bX_n)^T$ be a random sample from a $p$-dimensional normal distribution with mean $\zerovec$ and a positive definite covariance matrix $\bSigma$. The corresponding precision matrix, or the inverse covariance matrix $\bOmega = (\!(\omega_{ij})\!)$ is defined as $\bOmega = \bSigma^{-1}$. The natural estimator of $\bSigma$ is $\bmS = n^{-1}\sum_{i=1}^{n}\bX_i \bX_i^T$. We assume that $\bOmega$ is sparse, in the sense that the number of non-zero off-diagonal elements is small. We utilize the duality between a Bayesian prior and penalty, where the penalized likelihood estimate is understood to correspond to the maximum a posteriori (MAP) estimate under a given prior. Hence, for fully Bayesian inference on $\bOmega$, we need a suitable prior that also results in a penalty function with good frequentist properties; a non-trivial problem even in linear models \citep{castillo2015bayesian}. We put independent horseshoe-like priors \citep{bhadra2019horseshoe} on the off-diagonal and non-informative priors on the diagonal elements of $\bOmega$, while restricting the prior mass to positive definite matrices. A key benefit of the horseshoe-like prior, which closely mimics the sparsity-inducing global-local horseshoe prior \citep{carvalho2010horseshoe} is that the prior density, and hence the penalty, is available in closed form under the former, unlike under the latter. This allows one to study the penalty (equivalently, the negative logarithm of the prior density) directly, and to establish important properties concerning convexity (see, e.g., Lemma~\ref{lemma:concave}), which remain much more difficult under the horseshoe prior.
	For the fully Bayesian model, the element-wise prior specification induced by the horseshoe-like prior is,
	\begin{eqnarray}
		\label{eqn:prior-1}
		\omega_{ij}\mid a &\sim & \pi(\omega_{ij} \mid a),\;1 \leq i < j \leq p; \qquad
		\omega_{ii} \propto  1,\; 1 \leq i \leq p,
	\end{eqnarray}
	where $\pi(\omega_{ij} ; a) = \log \left( 1 + a/\omega_{ij}^2\right)/(2\pi a^{1/2})$ gives the horseshoe-like density function for $\omega_{ij}$. The motivation for using this density is two-fold: it has a sharp spike near zero, encoding the Bayesian prior belief  that most signals are ignorable; and it also possesses very heavy, polynomially decaying tails, allowing for identification of signals. These two properties closely mimic the popular horseshoe prior for sparse signals \citep{carvalho2010horseshoe}, and, in fact, one achieves the same origin and tail rates for the density function in terms of $\omega_{ij}$ as in the original horseshoe. The crucial advantage with the horseshoe-like, then, is that there is a closed form to the density function, unlike the horseshoe prior. Nevertheless, similar to the original horseshoe prior, the horseshoe-like prior also admits a convenient latent variable representation as a Gaussian scale-mixture \citep{bhadra2019horseshoe}. To be precise, one can write,
	\begin{equation}
		\label{prior}
		\omega_{ij} \mid \nu_{ij},a \sim  \NormRV\left( 0, \dfrac{a}{2\nu_{ij}}\right) ,\; \pi(\nu_{ij})  = \frac{1-\exp(-\nu_{ij})}{2\pi^{1/2}\nu_{ij}^{3/2}},
	\end{equation}
	where marginalizing over the latent $\nu_{ij}$ leads to the desired $\pi(\omega_{ij} ; a)$ identified above. For modeling valid precision matrices, we must restrict the prior mass on the space of symmetric positive definite matrices $\mathcal{M}_p^+$. Combining the unrestricted prior as in (\ref{eqn:prior-1}) and (\ref{prior}), along with the above restriction, leads to the joint prior specification on $\bOmega$ as,
	\begin{equation}
		\label{prior_2}
		\pi(\bOmega \mid \nu; a)\pi(\nu) \propto \underset{i,j: i < j} {\prod}\left\{1-\exp(-\nu_{ij})\right\}\nu_{ij}^{-1}\exp\left(\frac{-\nu_{ij}\omega_{ij}^2}{a}\right)\Ind_{\mathcal{M}_p^+}(\bOmega),
	\end{equation}
	where $\nu = \{\nu_{ij}\}_{i < j}$. In this formulation, the latent parameters $\nu_{ij}$ are component-specific, or local, and the shared parameter $a$ is global, situating the horseshoe-like in the broader category of global-local priors \citep{bhadra2019lasso}. Further details on the induced marginal prior on $\bOmega$ are presented in Supplementary Section~\ref{supp-ghs-prior-details}. Although it is possible to put a further hyperprior on $a$, it is considered fixed for point estimation approaches, and is estimated by the effective model size approach of \cite{piironen2017sparsity} to avoid a collapse to zero. We defer the details to Supplementary Section~\ref{global_scale_param_computation}. With the prior specification as in (\ref{prior_2}), the $\log$-posterior $\mathcal{L}$ thus becomes,
	\begin{equation}
		\label{likeli}
		\mathcal{L} \propto \frac{n}{2}\log \det \bOmega -\frac{n}{2}\text{tr}(\bmS\bOmega) + \underset{i,j: i<j} {\sum} \left\{\log\left(1-\exp(-\nu_{ij})\right) -\log \nu_{ij} - \frac{\nu_{ij}\omega_{ij}^2}{a} \right\}.
	\end{equation}
\begin{figure}[!t]
	\includegraphics[height=3.81cm,width=\textwidth]{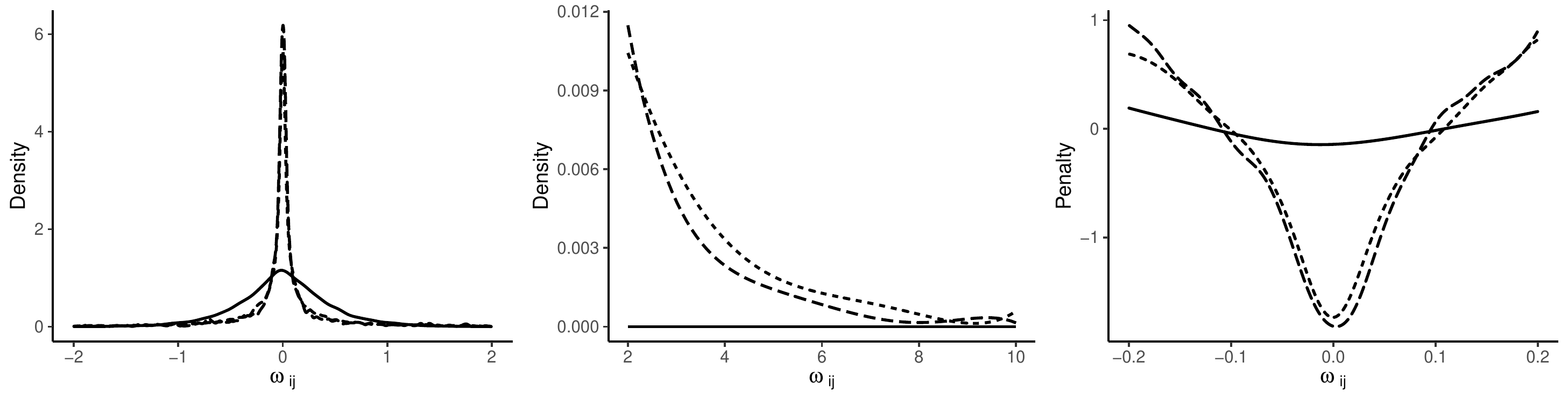}
	\caption{Smoothed density estimates of a randomly chosen off-diagonal element based on $10^4$ Markov chain Monte Carlo samples for the graphical horseshoe-like (dashes), the graphical horseshoe (small dashes) and the Bayesian graphical lasso (solid) priors; providing a visual comparison of (left) spikes near the origin, (middle) heaviness of the tails, and (right) the induced penalty functions for $p=10$. \label{fig:priors}}
\end{figure}

	At this point, the corresponding hierarchy of the horseshoe prior, which the horseshoe-like closely approximates, is well worth mentioning. The horseshoe prior \citep{carvalho2010horseshoe}, recognized as a state-of-the-art for sparse signal recovery \citep{bhadra2019horseshoe}, was deployed in estimating GGMs by \citet{li2019graphical} with the following hierarchy:
	\begin{equation}
		\label{hs_prior}
		\omega_{ij} \mid \lambda_{ij},\tau \sim  \NormRV\left( 0, \lambda_{ij}^2\tau^2\right) ,\; \pi(\lambda_{ij}^2)  \sim \mathcal{C}^+(0,1) ,\; \pi(\tau^2) \sim \mathcal{C}^+(0,1),
	\end{equation}
	where $\mathcal{C}^+(0,1)$ denotes the standard half-Cauchy distribution. It is recognized as the state of the art in sparse signal recovery in linear models \citep{bhadra2019horseshoe}, and was successfully deployed in estimating sparse precision matrices by \citet{li2019graphical}. Figure \ref{fig:priors} plots the smoothed histogram of prior densities of a randomly chosen off-diagonal element near the origin and at the tails for the graphical horseshoe-like along with two of its relatives: the graphical horseshoe  \citep{li2019graphical} and the Bayesian graphical lasso \citep{wang2012bayesian}. The corresponding penalties, given by the negative of the logarithm of the densities, are also shown. The key observations are: (a) the graphical horseshoe and graphical horseshoe-like densities are very similar and (b) both have far sharper spikes near the origin and heavier tails compared to the Laplace priors used in the Bayesian graphical lasso, providing an intuitive basis for the superiority of the horseshoe-family in sparse signal recovery. Extensive formal support for these observations are available in linear models \citep{bhadra2019lasso}, but barring some empirical evidence, the corresponding theoretical support is lacking in graphical models.
	
	Some comments on the desirability of a Gaussian scale mixture representation are also in order. First, the latent mixing variables make it easier to derive fully Bayesian computational strategies via data augmentation. A similar observation is true for point estimates via the expectation--maximization algorithm. Second, using a result of \citet{barndorff1982normal}, it is possible to derive a precise connection between the densities of the mixing variables and that of the resultant mixture. In particular, if the mixing densities are regularly varying in the tails, then so is the resultant Gaussian mixture. Since regular variation is closed under many nonlinear transformations, the heavier tails of global-local priors impart crucial robustness properties for estimating nonlinear, many-to-one functions of the parameters of interest in multi-parameter problems \citep{bhadra2015default}, and help avoid marginalization paradoxes of \citet{dawid1973marginalization}. 
	
\section{Estimation Procedure}
	\subsection{ECM Algorithm for MAP Estimation}
	\label{EM_algo}
	We utilize the Gaussian mixture representation of the horseshoe-like prior with latent scale parameters to devise an Expectation Conditional Maximization (ECM) \citep{meng93} approach to MAP estimation, building on the calculations for linear models by \citet{bhadra2019horseshoe}. For updating the elements of the precision matrix, we use the coordinate descent technique proposed by \citet{wang2014coordinate}, which guarantees the positive definiteness of the precision matrix at each update.
	\medskip
	\newline
	\textbf{E Step:} Following \citet{bhadra2019horseshoe}, we calculate the conditional expectation of the latent variable $\nu_{ij}, 1\leq i < j \leq p$, at current iteration $(t)$ as follows:
	\begin{equation}
		\label{e_step}
		\nu_{ij}^{(t)} = \mathbb{E}(\nu_{ij} \mid \omega_{ij}^{(t)}, a)  =\left(\log\left( 1+ \frac{a}{(\omega_{ij}^{(t)})^2}\right)\right)^{-1} \frac{a^2}{\left((\omega_{ij}^{(t)})^2 + a\right)\left( (\omega_{ij}^{(t)})^2\right)}.
	\end{equation}
	\textbf{CM Steps:} Having updated the latent parameters in the E-Step, the coordinate descent approach of \cite{wang2014coordinate} is used to update one column of the precision matrix at a time. Without loss of generality, we present the steps for updating the $p$th column. First we divide the precision matrix $\bOmega$ and the sample covariance matrix $\bmS$ into blocks as follows:
	\begin{center}
		\begin{tabular}{ c c }
			$\bOmega = \begin{bmatrix}
				\bOmega_{11} & \bOmega_{12}\\
				\bOmega_{12}^T & \Omega_{22}\\
			\end{bmatrix}$,  & 
			$ n\bmS = \begin{bmatrix}
				\bmS_{11} & \bmS_{12}\\
				\bmS_{12}^T & S_{22}\\
			\end{bmatrix}$,
		\end{tabular}
	\end{center}
	where, $\bOmega_{11}$ is a matrix of dimension $(p-1)\times(p-1)$ of the top left block of $\bOmega$; $\Omega_{22}$ is the $p$th diagonal element and $\bOmega_{12}$ is a $(p-1)\times 1$ dimensional vector of the remaining elements in the $p$th column. The decomposition of $n\bmS$ is analogous. We define $\gamma = \Omega_{22} - \bOmega_{12}^T\bOmega_{11}^{-1}\bOmega_{12}$ and $\bbeta = \bOmega_{12}$. With these transformations, we simplify  (\ref{likeli}) to update the $p$th column. We have,
	\begin{equation}
		\label{simpli_wang_coord_dec}
		\begin{split}
			\log \det \bOmega  & = \log (\gamma) +c_1, \\
			\text{tr}(n\bmS\bOmega) &= 2\bmS_{12}^T\bbeta + S_{22}\gamma + S_{22}\bbeta^T\bOmega_{11}^{-1}\bbeta + c_2,\\
			{\sum_{i,j: i<j}}\frac{\nu_{ij}^{(t)}}{a}\cdot \omega_{ij}^2 & = \bbeta^T\Lambdamat^{(t)}\bbeta + c_3,\\
			\Lambdamat^{(t)} & = \frac{1}{a} \mathrm{diag}\left(  \nu_{1p}^{(t)}, \ldots, \nu_{p-1,p}^{(t)}\right),
		\end{split}
	\end{equation}
	where $c_1,\text{ }c_2,\text{ }c_3$ are constants independent of $\bbeta,\text{ }\gamma$. Now the log-posterior with the transformed variables is given by,
	\begin{equation*}
		\mathcal{L} \propto \frac{n}{2}\log (\gamma) -\frac{1}{2}\left(2\bmS_{12}^T\bbeta + S_{22}\gamma + S_{22}\bbeta^T\bOmega_{11}^{-1}\bbeta \right) - \bbeta^T\Lambdamat^{(t)}\bbeta.
	\end{equation*}
	\begin{algorithm}[!t]
		\caption{ECM algorithm for MAP estimation (GHS-LIKE-ECM)}\label{em_algo_pseudo_code}
		\begin{algorithmic}
			\Function{ECM for Graphical Horseshoe-Like penalty}{$\bOmega_s, \bmS, n , p$} 
			\State $\bOmega_s = (\!(\omega_{s,ij})\!)$: starting point; $\bOmega_u = (\!(\omega_{u,ij})\!)$: updated precision matrix; initially set to $\bmI_p$.
			\State $\bmS=\bX^{(n)T}\bX^{(n)}/n$; $\bm{N} = (\!(N_{ij})\!)$: A matrix of dimension $p\times p$ which stores
			\State  $\mathbb{E}(\nu_{ij} \mid  \omega_{s,ij}, a)$; $(n,p)$: Sample size and number of variables respectively 
			\While{$\bDelta =  \|\bOmega_u - \bOmega_s \|_2 < \text{tolerance } (=10^{-3})$}
			\For{$j = 2\to p$}
			\For{$i = 1\to (j-1)$}
			\State $N_{ij}\text{ }:= \mathbb{E}(\nu_{ij} \mid \omega_{s,ij}, a)  = \left(\log\left( 1+ \frac{a}{(\omega_{s,ij})^2}\right)\right)^{-1} \frac{(a)^{2}}{\left((\omega_{s,ij})^2 + a\right)\left( (\omega_{s,ij})^2\right)}. $
			\EndFor
			\EndFor
			\State $\bm{N} \leftarrow \bm{N} +\bm{N}^T$. This is required to compute $\Lambdamat^{(t)}$ in display (\ref{simpli_wang_coord_dec}); Set $\bOmega_u = \bOmega_s$
			\For{$i= 1\leftarrow p$}
			\State Update $i^{th}$ column of $\bOmega_u$ using coordinate descent algorithm of \citet{wang2014coordinate} described above.
			\EndFor
			\EndWhile
			\State \Return $\hat\bOmega^{\mathrm{MAP}}=\bOmega_u$
			\EndFunction
		\end{algorithmic}
	\end{algorithm}
	Maximizing the above over $\bbeta, \gamma$ gives the required update as:
	\begin{equation}
		\label{update_beta_gamma}
		\hat{\gamma} =  \frac{n}{S_{22}},\quad
		\hat{\bbeta}  = -({S_{22}\bOmega_{11}^{-1}} + 2\cdot \Lambdamat^{(t)})^{-1}\bmS_{12}^T.
	\end{equation}
Having updated $\bbeta,\text{ }\gamma$ from (\ref{update_beta_gamma}), the $p$th column update of the precision matrix for the next iteration $(t+1)$ becomes
	\begin{equation}
		\label{pth_col_update}
		\hat\bOmega_{12}^{(t+1)}  = \hat{\bbeta},\quad
		\hat\bOmega_{12}^{T(t+1)} = \hat{\bbeta}^T,\quad
		\hat\Omega_{22}^{(t+1)}  = \hat{\gamma} + \hat{\bbeta}^T\bOmega_{11}^{-1}\hat{\bbeta}.
	\end{equation}
	We repeat the above steps for the remaining $(p-1)$ columns to complete the CM Step updates for $\bOmega$, until convergence to the MAP estimator $\hat{\bOmega}^{\mathrm{MAP}}$. The procedure is summarized in Algorithm~\ref{em_algo_pseudo_code}. The most computationally expensive step is the required inverse of a $(p-1)\times (p-1)$ matrix to compute $\hat\bbeta$ in (\ref{update_beta_gamma}), which needs to be repeated for each of the $p$ columns, giving a per iteration complexity of $O(p^4)$ for the algorithm. 
	\subsection{Posterior Sampling for the Fully Bayesian Estimate}
	\label{subsec:GHS-MCMC}
	For fully Bayesian estimation, we also outline the MCMC sampling procedure. With substitutions $2\nu_{ij} \mapsto t_{ij}^2$ and $a\mapsto \tau^2$, the prior in (\ref{prior}) can be written with a different hierarchy as follows:
	\begin{equation}
		\omega_{ij} \mid \nu_{ij},\tau \sim  \mathcal{N}\left( 0, \tau^2/t_{ij}^2\right) ,\; \pi(t_{ij})  =   \frac{1 -\exp\left(-t_{ij}^2/2\right)}{(2\pi)^{1/2}t_{ij}^2} ,\; t_{ij} \in \mathbb{R} ,\; \tau^2 > 0,  \nonumber
	\end{equation}
	where $\pi(t_{ij})$ above is known as the the slash normal density, expressed as $(\phi(0) - \phi(t_{ij}))/t_{ij}^2$, where $\phi(\cdot)$ is the standard normal density \citep{bhadra2019horseshoe}. Introducing a further local latent variable $r_{ij}$, the density for $t_{ij}$  can also be written as a normal scale mixture, where the scale follows a Pareto distribution, that is,
	\begin{equation*}
		t_{ij} \mid r_{ij} \sim \mathcal{N}(0, r_{ij}) ,\; r_{ij} \sim \text{Pareto}\left(1/2\right). 
	\end{equation*}
	For efficient sampling, the above Pareto scale mixture can be represented as a product of an exponential density and an indicator function as follows:
	\begin{equation}
		\begin{split}
			\pi(t_{ij}) & = \frac{1}{2}\int_{1}^\infty \frac{1}{(2\pi r_{ij})^{1/2}}\exp\left(-\frac{t_{ij}^2}{2r_{ij}}\right)r_{ij}^{-3/2}dr_{ij} = \frac{1}{2(2\pi)^{1/2}}\int_0^1 \exp\left(-\frac{t_{ij}^2m_{ij}}{2}\right)dm_{ij}, \\ \nonumber
			\text{i.e., } \pi(t_{ij}, m_{ij}) & = \frac{1}{2(2\pi)^{1/2}}\exp\left(\frac{-t_{ij}^2 m_{ij}}{2}\right)\Ind(0<m_{ij}<1).
		\end{split}    
	\end{equation}
	Different choices of prior for the global scale parameter are possible, but we consider $\tau \sim \mathcal{C}^{+}(0,1)$. \citet{makalic2015simple} observed that: if $\tau^2 \mid \xi \sim \text{InvGamma}(1/2,1/\xi)$ and $\xi \sim \text{InvGamma(1/2,1)}$ then marginally $\tau \sim \mathcal{C}^{+}(0,1)$. Using this, we can write the posterior updates of $\tau,\text{ }\xi$ as follows:
	\begin{equation}
		\label{scale_mix_cauchy_sampler}
		\begin{split}
			\tau^2 \mid \xi, \bX^{(n)}, \bOmega, \{t_{ij}\}_{i<j}, \{m_{ij}\}_{i<j}\ &\  \sim\  \text{InvGamma}((p(p-1)/2+1)/2,\, 1/\xi+\sum_{i,j:i<j}{t_{ij}^2\omega_{ij}^2/2}),\\
			\xi \mid \tau\ &\  \sim\  \text{InvGamma}(1+1/\tau^2).
		\end{split}
	\end{equation}
	Following the remaining updates from the graphical horseshoe sampler of \citet{li2019graphical}, the complete MCMC scheme for the graphical horseshoe-like is as outlined in Algorithm~\ref{alg:GHS}. The per iteration complexity of the algorithm is $O(p^4)$. Diagnostic plots for both the ECM and MCMC algorithms are given in Supplementary Section~\ref{sec:simulation_study_2}.
	
	\begin{algorithm}[!t]
		\caption{The Graphical Horseshoe-Like MCMC Sampler (GHS-LIKE-MCMC)}
		\label{alg:GHS}
		\begin{algorithmic}
			\Function{GHS-Like}{$\bmS,n,burnin,nmc$}
			\Comment{where $\bmS=\bX^{(n)T}\bX^{(n)}/n$, $n$ = sample size}
			
			\State Set $p$ to be number of rows (or columns) in $\bmS$
			\State Set initial values $\bOmega = \bmI_{p},\, \bSigma = \bmI_{p},\, \bm{T} = \bm{1}^T\bm{1},\, \bm{M} = \bm{1}^T\bm{1},\, \tau=1$, where $\bm{1}$ is a $p$-dimensional vector with all elements equal to 1, $\bm{T} = (\!(t_{ij}^2)\!)$, $\bm{M} = (\!(m_{ij})\!)$.	
			
			\For{$iter = 1$ to $(burnin+nmc)$}
			\For{$i = 1$ to $p$}
			
			\State $\gamma \sim \text{Gamma}(\text{shape}=n/2+1,\, \text{rate}=2/{s}_{ii})$  \Comment{sample $\gamma$}
			\State $\bOmega_{(-i)(-i)}^{-1} = \bSigma_{(-i)(-i)} - \bm{\sigma}_{(-i)i}\bm{\sigma}_{(-i)i}'/{\sigma}_{ii}$
			\State $\bm{C} = \left[{s}_{ii}\bOmega_{(-i)(-i)}^{-1}+\left(\text{diag}\left(\tau^2/t_{(-i)i}\right)\right)^{-1}\right]^{-1}$
			\State $\bm{\beta} \sim \text{Normal}(-\bm{C}\bmS_{(-i)i},\bm{C})$  \Comment{sample $\bm{\beta}$}
			\State $\bm{\omega}_{(-i)i} = \bm{\beta},\, {\omega}_{ii} = \gamma + \bm{\beta}^T\bOmega_{(-i)(-i)}^{-1}\bm{\beta}$  \Comment{variable transformation}
			
			\State $\bm{t}_{(-i)i} \sim \text{Gamma}(\text{shape}=3/2, \text{rate}= m_{(-i)i}/2+\bm{\omega}_{(-i)i}^2/2\tau^2)$  \Comment{sample $t$, where $\bm{t}_{(-i)i}$ is a vector of length $(p-1)$ with entries $t_{ji}^2, j \ne i$}
			\State $\bm{m}_{(-i)i} \sim \text{Exponential}(\text{rate = } \bm{t}_{(-i)i}/2)\text{ }\Ind(0< \bm{m}_{(-i)i}<1)$
			\Comment{sample $m$}
			
			\State Save updated $\bOmega$
			\State $\bSigma_{(-i)(-i)}=\bOmega_{(-i)(-i)}^{-1}+(\bOmega_{(-i)(-i)}^{-1}\bm{\beta})(\bOmega_{(-i)(-i)}^{-1}\bm{\beta})'/\gamma,\, \bm{\sigma}_{(-i)i}=-(\bOmega_{(-i)(-i)}^{-1}\bm{\beta})/\gamma,\, {\sigma}_{ii}=1/\gamma$
			\State Save updated $\bSigma, \bm{T},\bm{M}$.				
			\EndFor
			\State Sample $\tau,\text{ }\xi$ as in (\ref{scale_mix_cauchy_sampler}). \Comment{Sample $\tau,\text{ }\xi$}
			\EndFor
			
			\State Return MC samples $\bOmega$
			
			\EndFunction
		\end{algorithmic}
	\end{algorithm}
	
\section{Theoretical Properties}
	\subsection{Posterior Concentration Results}
	\label{sec:incons}
	In this section, we present our main result on the posterior contraction rate of the precision matrix $\bOmega$ around the true precision matrix $\bOmega_0$ with respect to the Frobenius norm under the graphical horseshoe-like prior. The technique of our proofs uses the general theory on posterior convergence rates as outlined in \cite{ghosal2000}, which establishes the desired convergence with respect to the Hellinger distance. However, from the perspective of a precision matrix, the Frobenius norm is easier to interpret in comparison to the Hellinger distance. Under suitable assumptions on the eigenspace of the precision matrices, \cite{banerjee2015bayesian} showed that these two distances are equivalent, and hence the same posterior contraction rates hold with respect to the Frobenius norm as well. We assume that the true underlying graph is sparse, so that the corresponding true precision matrix $\bOmega_0$ has $s$ non-zero off-diagonal elements. The total number of non-zero elements in $\bOmega_0$ is $p+s$, which gives the effective dimension of the parameter $\bOmega_0$. To establish the desired posterior concentration results, we shall need to control both the actual dimension and the effective dimension of the true precision matrix. Overall, our theoretical analyses depend on certain assumptions on the true precision matrix, the dimension and sparsity, and the prior space. We present the details of these assumptions along with relevant discussions below.
	
	\begin{assumption}\label{ass:prior_space}
		The prior is restricted to a subspace of symmetric positive definite matrices, $\mathcal{M}_p^+(L)$, where  
		\begin{equation}
			\label{eqn:prior-class}
			\mathcal{M}^+_p(L) = \{\bOmega \in \mathcal{M}_p^+: 0 < L^{-1} \leq \eig_1(\bOmega) \leq \cdots \leq \eig_p(\bOmega) \leq L < \infty\}.
		\end{equation}
		
	\end{assumption}
	
	\begin{assumption}\label{ass:1}
		The actual dimension $p$ satisfies the condition $p = n^b,\,b \in (0,1)$, and the effective dimension $p+s$ satisfies $(p+s)\log p/n = o(1).$
	\end{assumption}
	
	\begin{assumption}\label{ass:2}
		The true precision matrix $\bOmega_0$ belongs to the parameter space given by
		\begin{equation*}
			\mathcal{U}(\varepsilon_0, s) = \{\bOmega \in \mathcal{M}_p^+: \sum_{1\leq i < j \leq p}\Ind(\omega_{ij} \neq 0) \leq s, \, 0 < \varepsilon_0^{-1} \leq \eig_1(\bOmega) \leq \cdots \leq \eig_p(\bOmega) \leq \varepsilon_0 < \infty\}.
		\end{equation*} 
	\end{assumption}

	\begin{assumption}\label{ass:3}
		The bound $[L^{-1}, L]$ on the eigenvalues of $\bOmega$ as specified in (\ref{eqn:prior-class}) satisfies $L > \varepsilon_0$, or, in other words, $\varepsilon_0 = cL,$ for some $c \in (0, 1).$
	\end{assumption}

	\begin{assumption}
		\label{ass4}
		The global shrinkage parameter $a$ satisfies the condition, $a^{1/2} < n^{-1/2}p^{-b_1}$, for some constant $b_1 > 0.$
	\end{assumption}
	
	The condition on the eigenvalues of $\bOmega$ as specified in Assumption~\ref{ass:prior_space} is necessary for arriving at the theoretical results involving the posterior convergence rate of $\bOmega$. In this paper, we assume that $L$ is a fixed constant, which can be large. However, this condition does not affect the practical implementation of our proposed method, and is used purely as a technical requirement, so that we only can work with $\bOmega$ and $\bSigma$ that are away from singular matrices. Beyond this, no structural assumptions such as decomposability are placed on either $\bOmega$ or $\bSigma$. Similar assumptions have been made in related works; see \cite{liu2019empirical} and \cite{lee2021beta}. Assumption~\ref{ass:1} implies that the dimension grows to infinity as the sample size $n \rightarrow \infty$, but at a slower rate than $n$. Additionally, the condition on the effective dimension ensures that the posterior convergence rate goes to zero as $n \rightarrow \infty$. Similar conditions are necessary in proving the contraction results in other related works, for example, see \cite{banerjee2015bayesian, liu2019empirical, lee2021beta}. Assumption~\ref{ass:2} implies that the true precision matrix $\bOmega_0$ is sparse, and has eigenvalues that are bounded away from zero or infinity. Similar conditions are common in the literature on large precision matrix estimation problems; see, for example, \cite{banerjee2014posterior,banerjee2015bayesian, liu2019empirical, lee2021beta}; among others. Assumption~\ref{ass:3} is crucial in learning the precision matrix in a high-dimensional framework. This condition ensures that $\bOmega_0 \in \mathcal{M}_p^+(L)$, that is, the prior space contains the true precision matrix, which is necessary in efficient learning of the same. Assumption~\ref{ass4} ensures that the prior puts sufficient mass around the true zero elements in the precision matrix. The condition on the global scale parameter $a$ is a sufficient one, and is required to obtain the desired posterior convergence rate. We present the main theoretical result for posterior convergence now. A proof can be found in Appendix~\ref{app-main-post-conv}.
	\begin{theorem}
		Let $\bX^{(n)} = (\bX_1,\ldots,\bX_n)^T$ be a random sample from a $p$-dimensional normal distribution with mean $\bm{0}$ and covariance matrix $\bSigma_0 = \bOmega_0^{-1},$ where $\bOmega_0 \in \mathcal{U}(\varepsilon_0, s)$. Consider the prior specification as given by (\ref{prior_2}). Under the assumptions on the prior as given in Assumptions (\ref{ass:prior_space})--(\ref{ass4}), the posterior distribution of $\bOmega$ satisfies
		\begin{equation}
			\mathbb{E}_0\left[\Pr\{\|\bOmega - \bOmega_0\|_2 > M\epsilon_n \mid \bX^{(n)} \} \right] \rightarrow 0,\nonumber
		\end{equation} 
		for $\epsilon_n = n^{-1/2}(p+s)^{1/2}(\log p)^{1/2}$ and a sufficiently large constant $M > 0.$
		\label{thm:main-post-conv}
	\end{theorem}
	\begin{corollary}
		\label{cor-post-conv}
		Under similar conditions as in Theorem~\ref{thm:main-post-conv} above, the posterior distribution of $\bOmega$ has the posterior convergence rate $\epsilon_n = n^{-1/2}(p+s)^{1/2}(\log p)^{1/2}$ around $\bOmega_0$ with respect to the Frobenius norm under the graphical horseshoe prior as specified in (\ref{hs_prior}).
	\end{corollary}
	A proof of Corollary~\ref{cor-post-conv} is in Supplementary Section ~\ref{app-cor-post-conv} and settles the question of posterior concentration for the graphical horseshoe which \citet{li2019graphical} did not address. The posterior convergence rate above directly compares with the rate of convergence of the frequentist graphical lasso estimator \citep{rothman2008sparse}, and is identical to the posterior convergence rates obtained by \cite{banerjee2015bayesian} and \citet{liu2019empirical}. However, our work is the first to address \emph{unstructured precision matrices}, apart of a mild assumption of sparsity, using computationally efficient \emph{continuous shrinkage priors}. This is at a contrast with previous theoretical analyses that imposed restrictive assumptions such as decomposability.
	\subsection{Properties of the MAP Estimator}
	\label{subsec:MAP-theory}
	
	The MAP estimator of $\bOmega$ can be found by maximizing the following objective function:
	\begin{eqnarray}
		\label{eq:q-omega}
		Q(\bOmega) = \log \pi(\bOmega \mid \bX^{(n)}) & = & \ell(\bOmega) + \sum_{i,j: i < j} \log \pi(\omega_{ij} \mid a) + C  \nonumber \\
		& = & \dfrac{n}{2}\left(\log \det{\bOmega}- \tr(\bmS\bOmega) \right) - \sum_{i, j: i < j} pen_{a}(\omega_{ij}) + C, 
	\end{eqnarray}
	where, $pen_{a}(\omega) = -\log \log ( 1 + a/\omega^2), \; a > 0,$ is the horseshoe-like penalty. We start by proving $pen_{a}(\omega)$ is strictly concave in the following lemma, with a proof in Supplementary Section~\ref{app-lemma-concave}. 
	\begin{lemma}\label{lemma:concave}
		The extended real-valued penalty function $pen_{a}(x) = -\log \log ( 1 + a/x^2), \; a > 0,$ is strictly concave for all $x \in dom(pen_{a})$, separately for $x > 0$ and $x<0$. 
	\end{lemma}
	
	A direct consequence of Lemma \ref{lemma:concave} is as follows. Let $\bOmega^{(t)}$ be the $t$th iterate of a local linear approximation (LLA) algorithm \citep{zou2008one}, that is, 
	\[
	\bOmega^{(t+1)} = \argmax \left\{ \ell(\bOmega) -  \sum_{i, j: i < j} pen'_a(\lvert \omega_{ij}^{(t)}\rvert ) ~ \lvert \omega_{ij} \rvert \right\}, \quad t = 1, 2, \ldots.
	\]
	Then Theorem 1 of \citet{zou2008one}, together with the strict concavity of horsehoe-like penalty function from Lemma \ref{lemma:concave}, guarantees that the LLA algorithm will satisfy an ascent property, that is, $Q(\bOmega^{(t+1)}) > Q(\bOmega^{(t)})$, and hence the LLA algorithm will be a special case of minorize--maximize algorithms. We now present the result on consistency of the MAP estimate using the graphical horseshoe-like prior via an ECM algorithm, with a proof in Appendix~\ref{app-map}.
	\begin{theorem}
		\label{thm:MAP}
		Under the conditions of Theorem~\ref{thm:main-post-conv}, the MAP estimator of $\bOmega$, given by $\hat{\bOmega}^{\mathrm{MAP}}$ is consistent, in the sense that
		\begin{equation*}
			\|\hat{\bOmega}^{\mathrm{MAP}} - \bOmega_0\|_2 = O_P(\epsilon_n),
		\end{equation*}
		where $\epsilon_n$ is the posterior convergence rate as defined in Theorem~\ref{thm:main-post-conv}.
	\end{theorem}
	
	The above result guarantees that the MAP estimator also converges to the true precision matrix $\bOmega_0$ at the same rate as the posterior convergence rate in the Frobenius norm. By triangle inequality, Theorem~\ref{thm:main-post-conv} and Theorem~\ref{thm:MAP} together imply that $\|\bOmega - \hat{\bOmega}^{\mathrm{MAP}}\|_2 = O_P(\epsilon_n),$ so that the posterior probability of an $\epsilon_n$-neighborhood around the MAP estimator with respect to the Frobenius norm converges to zero. This pleasing correspondence between the fully Bayesian and MAP estimates under the same prior--penalty dual is far from guaranteed, in the face of possible contradictions pointed out by \citet{castillo2015bayesian} for the lasso in linear models.


	\section{Numerical Experiments}
	\label{sec:Simulation_study_1}
	We compare the MAP and MCMC estimates under the horseshoe-based methods (GHS, GHS-LIKE-MCMC and GHS-LIKE-ECM) with two frequentist approaches: GLASSO, GSCAD and one Bayesian approach: the Bayesian GLASSO (BGL). We consider two problem dimensions: $(n,p)= (120, 100)$ and $(120, 200)$. For each dimension,  we perform simulations under four different structures of the true precision matrix $\bOmega_0$ as in \citet{li2019graphical} and \citet{friedman2008glasso}. These are: Random, Hubs, Cliques positive and Cliques negative, as detailed below.
	\begin{enumerate}
		\item \emph{Random.} The off-diagonal entries of $\bOmega_0$ are non-zero with probability $0.01$ when $p=100$ and $0.002$ when $p=200$. The non-zero entries are then sampled uniformly from $(-1,-0.2)$.
		
		\item \emph{Hubs.} The rows/columns are partitioned into $K$ disjoint groups $G_1, \ldots, G_K$. The off-diagonal entries $\omega_{ij}^0$ are set to 0.25 if $i\neq j$ and $i,j\in G_k$ for $k = 1,\dots,K$. In our simulations we consider $p/10$ groups with equal number of elements in each group. 
		
		\item \emph{Cliques positive and Cliques negative.} Same as Hubs, except for setting all $\omega_{ij}^0$, $i\neq j$ and $i,j\in G_k$, we select 3 members within each group, $g_k \subset G_k$, and set $\omega_{ij}^0 = 0.75$, $i\neq j$ and $i,j\in g_k$ for `Cliques positive' and set $\omega_{ij}^0 = -0.45$, $i\neq j$ and $i,j\in g_k$ for `Cliques negative.' 
	\end{enumerate}
	For each setting of $(n,p)$ and $\bOmega_0$, we generate 50 data sets and estimate the precision matrices by the methods stated above. All three horseshoe based methods are implemented in MATLAB, GSCAD is as implemented by \citet{wang2012bayesian} and GLASSO is implemented in 
	R package `glasso' \citep{R:glasso}.  Starting points for GHS-LIKE-ECM are randomly chosen in order to avoid getting stuck in a local minimum (see details in Supplementary Section \ref{sec:simulation_study_2}) and its global shrinkage parameter is chosen as in Supplementary Section~\ref{global_scale_param_computation}. Tuning parameters for GLASSO and GSCAD are chosen by 5-fold cross validation. The middle $50\%$ posterior credible intervals are used for variable selection for the Bayesian approaches. We provide resuts on: Stein's loss ($={\tr(\hat{\bOmega}\bSigma_0) - \log\det (\hat{\bOmega}\bSigma_0) - p}$), Frobenius norm ($\text{F norm}=\norm{\hat\bOmega - \bOmega_0}_2$), true and false positive rates for detecting non-zero off-diagonal entries (resp., TPR and FPR), the Matthews Correlation Coefficient (MCC), and average CPU time. Note that for the fully Bayesian estimate, our theory concerns posterior concentration properties, and connections with convergence in Frobenius norm is established in \citet{banerjee2015bayesian}. However, for the sake of completeness and comparisons with point estimation approaches, we provide variable selection results for all approaches as well, in addition to Stein's loss (an empirical measure of Kullback--Leibler divergence) and F norm that focus more directly on the entire distribution.
	
	It can be clearly seen from Tables \ref{simulation_p_100_n_120_part_1}-- \ref{simulation_p_200_n_120_part_2} that the horseshoe based methods generally perform the best. GHS has the smallest Stein's loss in all settings expect in Hubs when $(n,p) = (120,100)$. This corroborates the finding of \citet{li2019graphical} that GHS results in improved Kullback--Leibler risk properties (of which Stein's loss is an empirical measure) when compared to prior densities that are bounded above at the origin, e.g., BGL, and it is apparent from both tables that BGL has the worst Stein's loss. For GHS-LIKE-ECM and MCMC, the measures of Stein's loss are generally close to that of GHS, and much better compared to the other competing methods. A similar pattern emerges in the results for F norm, with the horseshoe-based methods once again outperforming the competitors and performing similarly among themselves. It is worth noting, however, that the fully Bayesian approaches (GHS-LIKE--MCMC and GHS) generally result in the best statistical performance, at the expense of a considerably longer computing time, making the trade-off between fully Bayesian and penalized likelihood approaches apparent.

		\begin{table}[!t]
		\centering
		\caption{Mean (sd) Stein's loss, Frobenius norm, true positive rates and false positive rates, Matthews Correlation Coefficient of precision matrix estimates over 50 data sets generated by multivariate normal distributions with precision matrix $\bOmega_0$, where $n=120$ and $p=100$. The precision matrix is estimated by frequentist graphical lasso with penalized diagonal elements (GL1) and with unpenalized diagonal elements (GL2), graphical SCAD (GSCAD), Bayesian graphical lasso (BGL), the graphical horseshoe (GHS),  graphical horseshoe-like ECM (ECM) and graphical horseshoe-like MCMC (MCMC). The best performer in each row is shown in bold. Average CPU time is in seconds.
		}
		\label{simulation_p_100_n_120_part_1}
		\begin{footnotesize}
				\begin{tabular}{l rrrrrrr }
					\hline
					{ }& \multicolumn{7}{c}{Random}  \\
					{ }& \multicolumn{7}{c}{35 nonzero pairs out of 4950}  \\
					{ }& \multicolumn{7}{c}{nonzero elements $\sim -\mathrm{Unif}(0.2,1)$} \\
					& GL1 & GL2 & GSCAD & BGL & GHS & ECM & MCMC  \\
					\hline\\
					Stein's loss & 5.245 & 6.785 & 5.21 & 42.997 & \textbf{2.176} & 3.758 & 2.626\\
					& (0.254) & (0.464) & (0.242) & (0.898) & (0.278) & (0.282) & (0.317)\\
					
					F norm & 3.348 & 4.084 & 3.333 & 3.952 & \textbf{1.194} & 2.224 & 2.164\\
					& (0.115) & (0.143) & (0.117) & (0.139) & (0.144) & (0.108) & (0.16)\\
					
					TPR & 0.951 & 0.882 & \textbf{0.998} & 0.979 & 0.819 & 0.948 & 0.856\\
					& (0.03) & (0.038) & (0.009) & (0.023) & (0.041) & (0.032) & (0.036)\\
					
					FPR & 0.101 & 0.045 & 0.994 & 0.166 & \textbf{0.0005} &  0.071 & 0.003\\
					& (0.013) & (0.007) & (0.005) & (0.0007) & (0.0003) & (0.005) & (0.001)\\
						
					MCC & 0.232 & 0.321 & 0.005 & 0.181 & \textbf{0.869} & 0.275 & 0.776\\
					& (0.018) & (0.024) & (0.001) & (0.007) & (0.031) & (0.016) & (0.034) \\
					
					Avg CPU time & 4.988 & 4.719 & 53.977 & 550.422 & 252.84 & 5.94 & 538.929 \\

					\hline
					
					{ }& \multicolumn{7}{c}{Cliques negative}\\
					{ }& \multicolumn{7}{c}{30 nonzero pairs out of 4950} \\
					{ }& \multicolumn{7}{c}{nonzero elements = -0.45} \\
					& GL1 & GL2 & GSCAD & BGL & GHS & ECM & MCMC  \\
					\hline\\
					Stein's loss & 4.607 & 7.134 & 4.567 & 42.618 & \textbf{1.862} & 3.417 & 2.334\\
					& (0.223) & (0.529) & (0.231) & (0.896) & (0.263) & (0.251) & (0.307)\\
					
					F norm 	& 2.823 & 3.851 & 2.813 & 3.814 & \textbf{1.969} & 2.107 & 2.132\\
					& (0.117) & (0.138)  & (0.112) & (0.165) & (0.212) & (0.124) & (0.199) \\
					
					TPR & \textbf{1} & \textbf{1} & \textbf{1} & \textbf{1} & 0.983 & \textbf{1} & 0.988 \\
					& (0) & (0) & (0) & (0) & (.024) & (0) & (0.02)\\
					
					FPR & 0.1 & 0.028 &  0.983 & 0.158 & \textbf{0.0004} & 0.073 & 0.002  \\
					& (0.01) & (0.006) & (0.013) & (0.007) & (0.0003) & (0.005) & (0.001)\\
					
					MCC & 0.232 & 0.42 & 0.01 & 0.177 & \textbf{0.936} & 0.268 & 0.843 \\
					& (0.014) & (0.036) & (0.003) & (0.005) & (0.024) & (0.009) & (0.037) \\
					
					Avg CPU time & 2.962 & 3.2648 & 24.792 & 550.768 & 253.04 & 5.282 & 540.928 \\	   
					\hline
			\end{tabular}
		\end{footnotesize}
	\end{table}
    	\begin{table}[!t]
		\centering
		\caption{Mean (sd) Stein's loss, Frobenius norm, true positive rates and false positive rates, Matthews Correlation Coefficient of precision matrix estimates over 50 data sets generated by multivariate normal distributions with precision matrix $\bOmega_0$, where $n=120$ and $p=100$. The precision matrix is estimated by frequentist graphical lasso with penalized diagonal elements (GL1) and with unpenalized diagonal elements (GL2), graphical SCAD (GSCAD), Bayesian graphical lasso (BGL), the graphical horseshoe (GHS),  graphical horseshoe-like ECM (ECM) and graphical horseshoe-like MCMC (MCMC). The best performer in each row is shown in bold. Average CPU time is in seconds.
		}
		\label{simulation_p_100_n_120_part_2}
		\begin{footnotesize}
				\begin{tabular}{l rrrrrrr }
					\hline
					{ }&  \multicolumn{7}{c}{Hubs} \\
					{ }&\multicolumn{7}{c}{90	nonzero pairs out of 4950} \\
					{ }& \multicolumn{7}{c}{nonzero elements = 0.25} \\
					& GL1 & GL2 & GSCAD & BGL & GHS & ECM & MCMC  \\
					\hline\\
					Stein's loss & 5.255 & 6.328 & 5.213 & 43.042 & 5.101 &\textbf{4.22} & 5.310 \\
				    & (0.263)  & (0.414)  & (0.261) & (0.802) & (0.455) & (0.369) & (0.485)\\
					
					F norm & 3.018 & 3.432 & 3.003 & 4.295 & 2.544 & \textbf{2.415} & 2.687\\
					& (0.091) & (0.112) & (0.093) & (0.156) & (0.126) & (0.103) & (0.141) \\
					
					TPR & 0.995 & 0.986 & \textbf{0.998} & 0.995 & 0.872 & 0.985 & 0.754 \\		
					& (0.007) & (0.017) & (0.002) & (0.008) & (0.04) & (0.014) & (0.004)\\
					
					FPR & 0.101 & 0.045 &  0.983 & 0.186 & \textbf{0.003} & 0.062 & \textbf{0.003}\\
					& (0.016) & (0.008) & (0.012) & (0.007) & (0.001) & (0.005) & (0.001)\\
					
					MCC & 0.373 & 0.523 & 0.016 & 0.27 & \textbf{0.85} & 0.458 & 0.775 \\
					& (0.027) & (0.039) & (0.006) & (0.006) & (0.027) & (0.015) & (0.033) \\
					
					Avg CPU time & 1.739 & 1.76 & 48.54 & 549.196 & 252.94 & 5.811 & 537.604\\
					
					\hline
					
					{ }& \multicolumn{7}{c}{Cliques positive} \\
					{ }& \multicolumn{7}{c}{30 nonzero pairs out of 4950} \\
					{ }&  \multicolumn{7}{c}{nonzero elements = 0.75} \\
					& GL1 & GL2 & GSCAD & BGL & GHS & ECM & MCMC \\
					\hline\\
					Stein's loss & 6.010 & 7.48 & 5.98 &44.163 & \textbf{1.781} & 3.753 & 2.425\\
					& (0.212) & (0.45) &(0.21) & (0.790) & (0.232) & (0.275) & (0.323)\\
					
					F norm 	& 4.96 & 5.7 & 4.95 & 4.916 & \textbf{1.888} & 2.411 & 2.170\\
					& (0.1) & (0.13) &(0.107) & (0.103) & (0.184) & (0.142) & (0.198)\\
					
					TPR & \textbf{1} &\textbf{1} & \textbf{1} & \textbf{1} & \textbf{1}& \textbf{1} & \textbf{1}\\
					& (0) & (0) & (0) & (0) & (0) & (0) & (0)\\
					
					FPR & 0.11 & 0.042 & 0.972 & 0.177&  \textbf{0.0008} & 0.068 & 0.003  \\
					& (0.013) & (0.0011) &  (0.014) & (0.006)  & (0.005) & (0.006) & (0.001)\\
					
					MCC & 0.22 & 0.353 & 0.013 & 0.166 & \textbf{0.94} & 0.277 & 0.814\\
					& (0.013) & (0.041) & (0.003) & (0.004) & (0.031) & (0.012) & (0.035) \\
					
					Avg CPU time & 1.997 & 2.157 & 83.852 & 553.743 & 252.46 & 5.903 & 539.046\\	   
					\hline
			\end{tabular}
		\end{footnotesize}
	\end{table}
  
    \begin{table}[!th]
		\centering
		\caption{Mean (sd) Stein's loss, Frobenius norm, true positive rates and false positive rates, Matthews Correlation Coefficient of precision matrix estimates over 50 data sets generated by multivariate normal distributions with precision matrix $\bOmega_0$, where $n=120$ and $p=200$. The precision matrix is estimated by frequentist graphical lasso with penalized diagonal elements (GL1) and with unpenalized diagonal elements (GL2), graphical SCAD (GSCAD), Bayesian graphical lasso (BGL), the graphical horseshoe (GHS),  graphical horseshoe-like ECM (ECM) and graphical horseshoe-like MCMC (MCMC). The best performer in each row is shown in bold. Average CPU time is in seconds.
		}
		
		\label{simulation_p_200_n_120_part_1}
		\begin{footnotesize}
				\begin{tabular}{l rrrrrrr }
					\hline
					{ }& \multicolumn{7}{c}{Random} \\
					{ }& \multicolumn{7}{c}{29 nonzero pairs out of 19900} \\
					{ }& \multicolumn{7}{c}{nonzero elements $\sim -\mathrm{Unif}(0.2,1)$}  \\
					& GL1 & GL2 & GSCAD & BGL & GHS & ECM & MCMC  \\
					\hline\\
					Stein's loss & 10.06 & 15.578 & 9.975 & 117.092 & \textbf{3.073} & 11.109 & 5.632\\
					& (0.4) & (1.12) & (0.4) & (1.563) & (0.305) & (0.562) & (0.523)\\
					
					F norm & 4.469 & 5.929 & 4.44 & 6.803 & \textbf{2.468} & 3.917 & 3.313\\
					& (0.151) & (0.176) & (0.156) & (0.162) & (0.137) & (0.108) & (0.178)\\
					
					TPR & 0.944 & 0.845 & \textbf{0.999} & 0.982 & 0.848 & 0.97 & 0.877\\		
					& (0.036) & (0.036) & (0.005) & (0.024) & (0.038) & (0.028) & (0.041)\\
					
					FPR & 0.052 & 0.163 & 0.984 & 0.103 & \textbf{0.0001} &  0.066 & 0.002\\
					& (0.007) & (0.002) & (0.011) & (0.003) & (0.00007) & (0.003) & (0)\\
					
					MCC & 0.152 & 0.242 & 0.004 & 0.11 & \textbf{0.882} & 0.138 & 0.599 \\
					& (0.011) & (0.015) & (0.002) & (0.004) & (0.029) & (0.005) & (0.035) \\
					
					Avg CPU time & 38.759 & 43.486 & 510.703 & 4484.22 & 1866.47 & 80.939 & 2260.3\\
					
					\hline
					
					{ }& \multicolumn{7}{c}{Cliques negative} \\
					{ }& \multicolumn{7}{c}{60 nonzero pairs out of 19900} \\
					{ }& \multicolumn{7}{c}{nonzero elements = -0.45}\\
					& GL1 & GL2 & GSCAD & BGL & GHS & ECM & MCMC  \\
					\hline
					\\
					Stein's loss & 11.604 & 18.088 & 11.541 & 125.138 & \textbf{3.985} & 12.467 & 6.179\\
					& (0.401) & (0.993) & (0.396) & (1.714) & (0.403) & (0.626) & (0.403) \\
					
					F norm 	& 4.443 & 6.024 & 4.439 & 6.299 &\textbf{ 2.861} & 3.8 & 3.331\\
					& (0.088) & (0.143) & (0.076) & (0.168) & (0.209) & (0.126) & (0.2) \\
					
					TPR & \textbf{1} & \textbf{1} & \textbf{1} & \textbf{1} & .975 & \textbf{1} & 0.991\\
					& (0) & (0) & (0) & (0) & (.173) & (0) & (0.012) \\
					
					FPR & 0.066 & 0.016 & 0.998 & 0.099 & \textbf{0.0002} & 0.084 & 0.002\\
					& (0.004) & (0.002) & (0.005) & (0.002) & (0.0001) & (0.003) & (0)\\
					
					MCC & 0.202 & 0.395 & 0.006 & 0.164 & \textbf{0.944} & 0.178 & 0.784 \\
					& (0.006) & (0.027) & (0.001) & (0.002) & (0.16) & (0.003) & (0.028) \\
					
					Avg CPU time & 32.936 & 36.49 & 548.26 & 4492.67 & 1876.96 & 70.683 & 2249.3\\	   
					\hline
			\end{tabular}
		\end{footnotesize}
	\end{table}
	\begin{table}[!th]
		\centering
		\caption{Mean (sd) Stein's loss, Frobenius norm, true positive rates and false positive rates, Matthews Correlation Coefficient of precision matrix estimates over 50 data sets generated by multivariate normal distributions with precision matrix $\bOmega_0$, where $n=120$ and $p=200$. The precision matrix is estimated by frequentist graphical lasso with penalized diagonal elements (GL1) and with unpenalized diagonal elements (GL2), graphical SCAD (GSCAD), Bayesian graphical lasso (BGL), the graphical horseshoe (GHS),  graphical horseshoe-like ECM (ECM) and graphical horseshoe-like MCMC (MCMC). The best performer in each row is shown in bold. Average CPU time is in seconds.
		}
		
		\label{simulation_p_200_n_120_part_2}
		\begin{footnotesize}
				\begin{tabular}{l rrrrrrr }
					\hline
					{ }&  \multicolumn{7}{c}{Hubs} \\
					{ }& \multicolumn{7}{c}{180	nonzero pairs out of 19900} \\
					{ }&  \multicolumn{7}{c}{nonzero elements = 0.25} \\
					& GL1 & GL2 & GSCAD & BGL & GHS & ECM & MCMC  \\
					\hline\\
					Stein's loss & 12.407 & 15.243 &12.331 & 123 & \textbf{11.692} & 12.825 & 12.741 \\
				    & (0.491) & (0.819) & (0.465) & (1.31) & (0.781) & (0.623) & (0.922)\\
					
					F norm & 4.594 & 5.3 & 4.583 & 7.129 & \textbf{3.763} & 4.209 & 4.179\\
					& (0.01) & (0.152) & (0.084) & (0.16) & (0.132) & (0.107) & (0.174)\\
					
					TPR & 0.99 & 0.976 & \textbf{1} & 0.991 & 0.779 & 0.986 &  0.737\\		
					& (0.007) & (0.137) & (0) & (0.006) & (0.034) & (0.009) & (0.033) \\
					
					FPR & 0.065 & 0.024 & 0.999 & 0.119 & \textbf{0.001} & 0.066 & 0.003\\
					& (0.005) & (0.006) & (0.001) & (0.003) & (0.0003) & (0.003) & (0)\\
					
					MCC & 0.336 & 0.515 & 0.01 & 0.248 & \textbf{0.82} & 0.332 & 0.722\\
					& (0.015) & (0.043) & (0.003) & (0.003) & (0.024) & (0.006) & (0.03) \\
					
					Avg CPU time & 17.847 & 19.917 & 517.33 & 4499.30 & 1870.57 & 74.808 & 2523.1\\
					
					\hline
					
					{ }&  \multicolumn{7}{c}{Cliques positive} \\
					{ }& \multicolumn{7}{c}{60 nonzero pairs out of 19900} \\
					{ }& \multicolumn{7}{c}{nonzero elements = 0.75} \\
					& GL1 & GL2 & GSCAD & BGL & GHS & ECM & MCMC  \\
					\hline
					\\
					Stein's loss & 14.523 & 17.262 & 14.477 & 126.487 & \textbf{3.797} & 13.512 & 6.717\\
					& (0.339) & (0.692) & (0.333) & (1.41) & (0.35) & (0.522) & (0.479) \\
					
					F norm 	& 7.59 & 8.553 & 7.596 & 7.936 & \textbf{2.733} & 4.248 & 3.535\\
					& (0.1) & (0.115) & (0.091) & (0.109) & (0.181) & (0.142) & (0.191) \\
					
					TPR & \textbf{1} & \textbf{1} & \textbf{1} & \textbf{1} & \textbf{1} & \textbf{1} & \textbf{1}  \\
					& (0) & (0) & (0) & (0) & (0) & (0) & (0) \\
					
					FPR & 0.065 & 0.024 & 0.991 & 0.115 & \textbf{0.0004} & 0.08 & 0.003\\
					& (0.005) & (0.004) & (0.007) & (0.002)  & (0.0001) & (0.002) & (0)\\
					
					MCC & 0.205 & 0.335 & 0.01 & 0.15 & \textbf{0.959} & 0.184 & 0.735 \\
					& (0.007) & (0.028) & (0.002) & (0.002) & (0.19) & (0.003) & (0.024) \\
					
					Avg CPU time & 23.768 & 25.3 & 880.46 & 4497.97 & 1872.55 & 80.652 & 2262.6\\	   
					\hline
			\end{tabular}
		\end{footnotesize}
	\end{table}
	Coming next to variable selection results, one may expect the penalized likelihood approaches to really shine; since these methods produce exact zeros, unlike the Bayesian approaches that necessitate some form of post-processing. Nevertheless, the Bayesian approaches offer the advantage of controlling the trade-off between TPR and FPR, by varying the width of the credible interval, for example. With our chosen mechanism (i.e., a variable is considered not to be selected if the middle 50\% credible interval includes zero), the GHS-LIKE-MCMC and GHS have the smallest TPR. Nevertheless, the penalized methods also have higher FPR in general (except for GHS-LKE-ECM), which results in lower MCC overall. In particular, the GSCAD estimate, which is not guaranteed to be positive definite in finite samples \citep{fan2016overview}, seems not to work well in general. Additional numerical results investigating the choice of starting values for the GHS-LIKE-ECM algorithm is given in Supplementary Section~\ref{sec:simulation_study_2}.
	
\section{Protein--Protein Interaction Network in B-cell Lymphoma}
	\label{sec:real_life_example}
	We analyze Reverse Phase Protein Array (RPPA) data of 33 patients with lymphoid neoplasm ``Diffuse Large B-cell Lymphoma'' to infer the protein interaction network. The data set consists of protein expressions for 67 genes across 12 pathways for all patients. 
	\begin{table}[!th]
		\centering
		\caption{Percentage of zeros ($\%$ Sparsity) and number of non-zero entries (NNZ) in the lower triangle of the precision matrix estimate of RPPA data for the competing approaches.} 
		\begin{tabular}{l  rrrrrrc} 
			\hline
			& MCMC & ECM & GHS & BGL & GL1 & GL2 & GSCAD \\ [0.5ex] 
			\hline
			$\%$ Sparsity & 95.79 & 88.6 & 91.59 & 73.72 & 69.88 & 73.67 & $9.06\times 10^{-4}$ \\ 
			NNZ & 93 & 252 &  186 & 581 & 666 & 582 & 2209\\ 
			\hline
		\end{tabular}
		\label{tab:real_data_nnz_and_sparsity}
	\end{table}
	As in simulations, we use $50\%$ posterior credible intervals for variable selection in GHS, BGL and GHS-LIKE-MCMC. The estimated sparsity (\% of zero elements) and number of non zeros in the lower triangle of the estimates are given in Table \ref{tab:real_data_nnz_and_sparsity}. We note that the GHS-LIKE-MCMC gives the sparsest estimate, almost 4\% sparser than the GHS. This is consistent with prior studies that found robust gene networks are typically sparse \citep{leclerc2008survival}.  As in the simulations, GSCAD performs the worst. To compare with a prior analysis of the same data set, we use the PRECISE framework of \citet{ha2018personalized}. This method can infer directed edges, but we ignore the directionality since we are interested in interactions and not causation. The proportions of edges in the estimates that `agree' and `do not agree' with the edges inferred using PRECISE framework are presented in Table \ref{tab:prop_AE_NE}. Protein networks realized from the estimates are presented in Figure \ref{graph_est_all_methods}. It can be seen that the GHS-LIKE-MCMC has the sparsest estimate among the methods that allow for interaction across all proteins, unlike the PRECISE framework that ignores interactions among proteins in different pathways, which may not be biologically justifiable.
	\begin{table}[!h]
		\centering
		\caption{Proportion of edges that `agree' (AE) and `do not agree' (NE) with the edges inferred using the PRECISE framework.} 
		\footnotesize{
		\begin{tabular}{ l  rrrrrrr} 
			\hline
			& MCMC & ECM & GHS & BGL & GL1 & GL2 & GSCAD \\ [0.5ex] 
			\hline
			AE & 0.238 & 0.412 & 0.325 & 0.575 & 0.638 & 0.6 & 1 \\ 
			NE & 0.034 & 0.101 &  0.074 & 0.247 & 0.284 & 0.247 & 0.984\\ 
			\hline
		\end{tabular}}
		\label{tab:prop_AE_NE}
	\end{table}
	\begin{figure}[!h]
		\centering 
		\begin{subfigure}{0.5\textwidth}
			\includegraphics[scale=0.6]{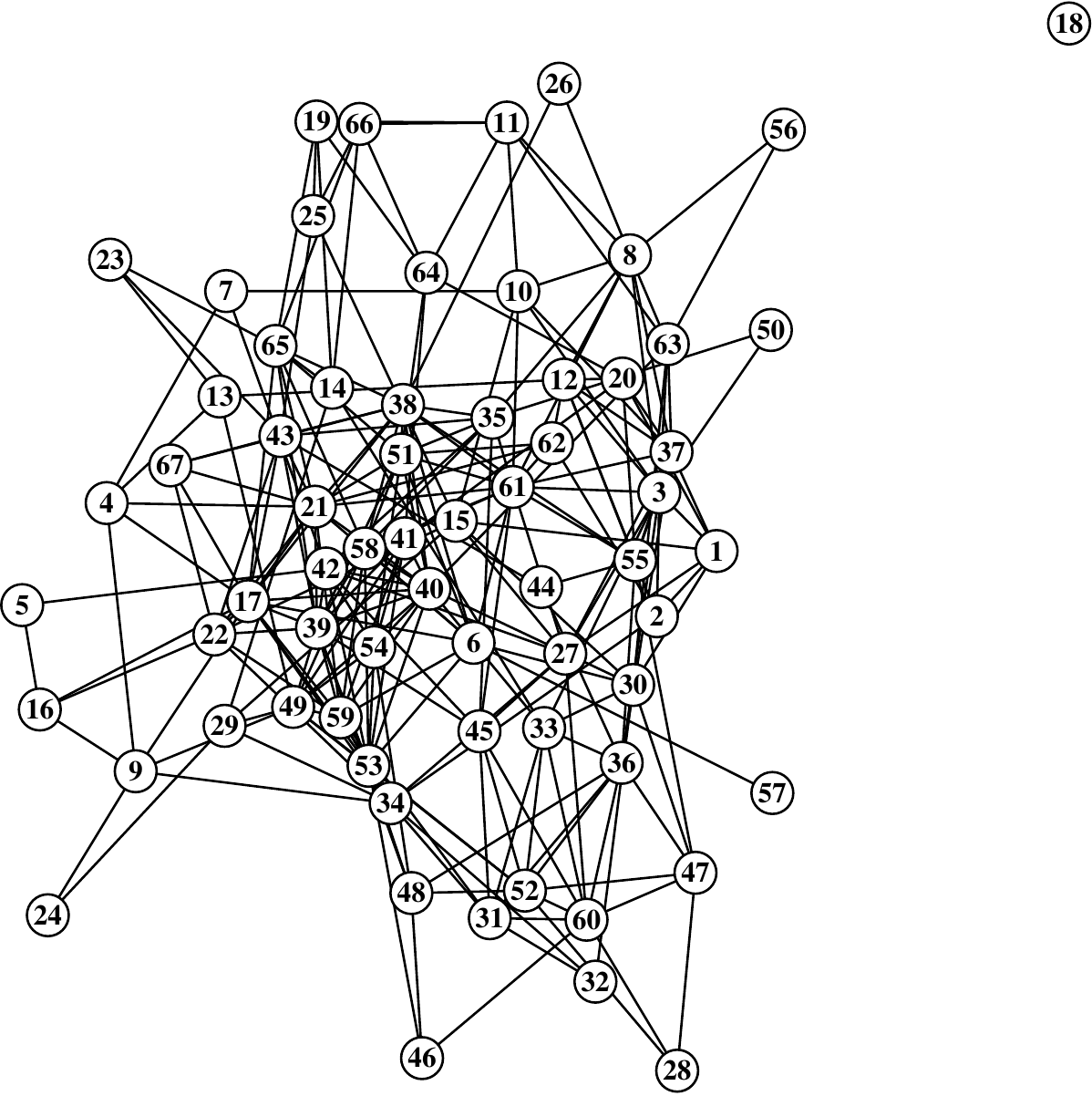}
			\caption{}
		\end{subfigure}\hfil 
		\begin{subfigure}{0.5\textwidth}
			\includegraphics[scale=0.6]{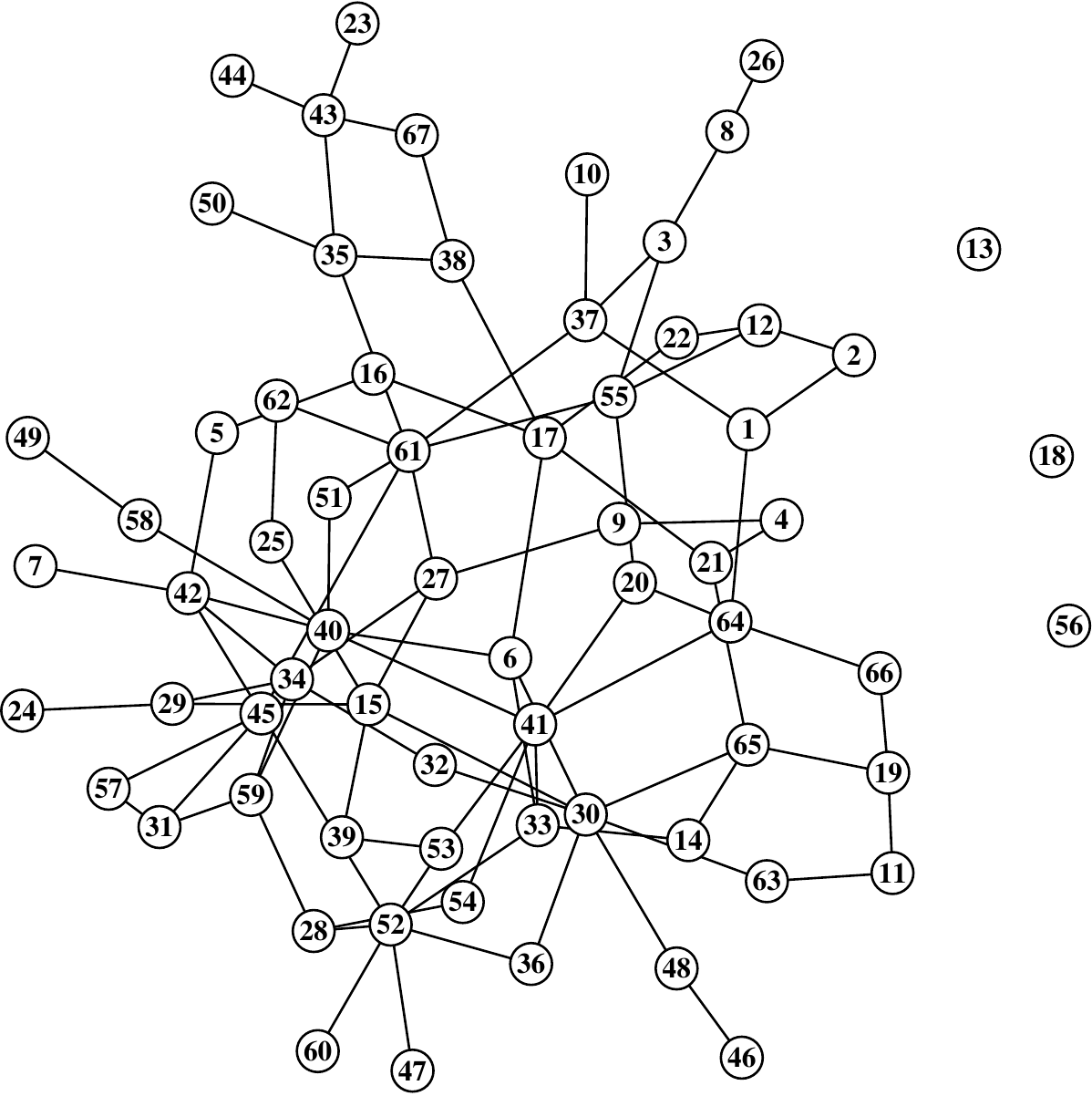}
			\caption{}
		\end{subfigure}\hfil 
		\medskip
		\medskip
		\begin{subfigure}{0.5\textwidth}
			\includegraphics[scale=0.6]{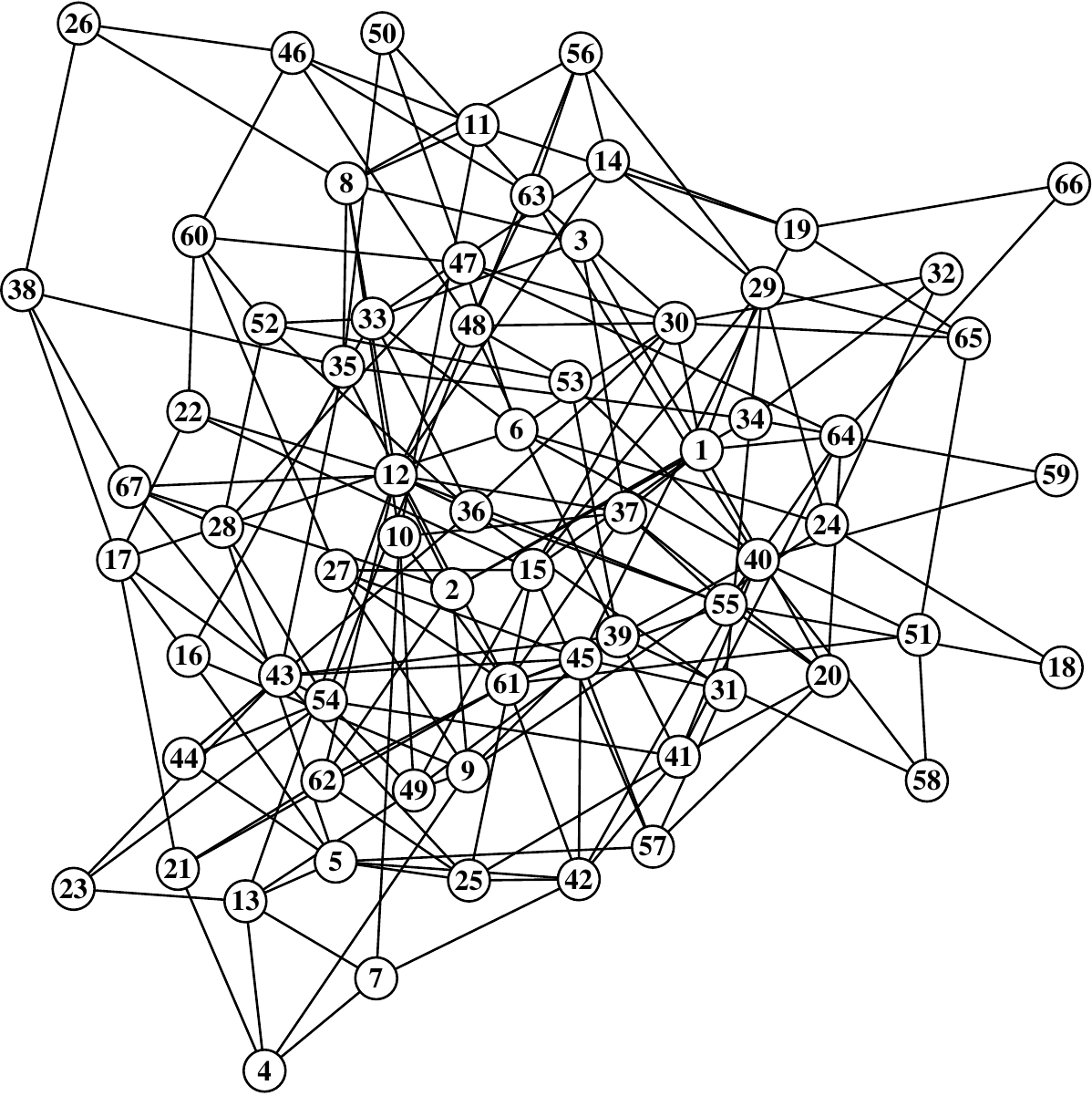}
			\caption{}
		\end{subfigure}\hfil 
		\begin{subfigure}{0.5\textwidth}
			\includegraphics[scale=0.6]{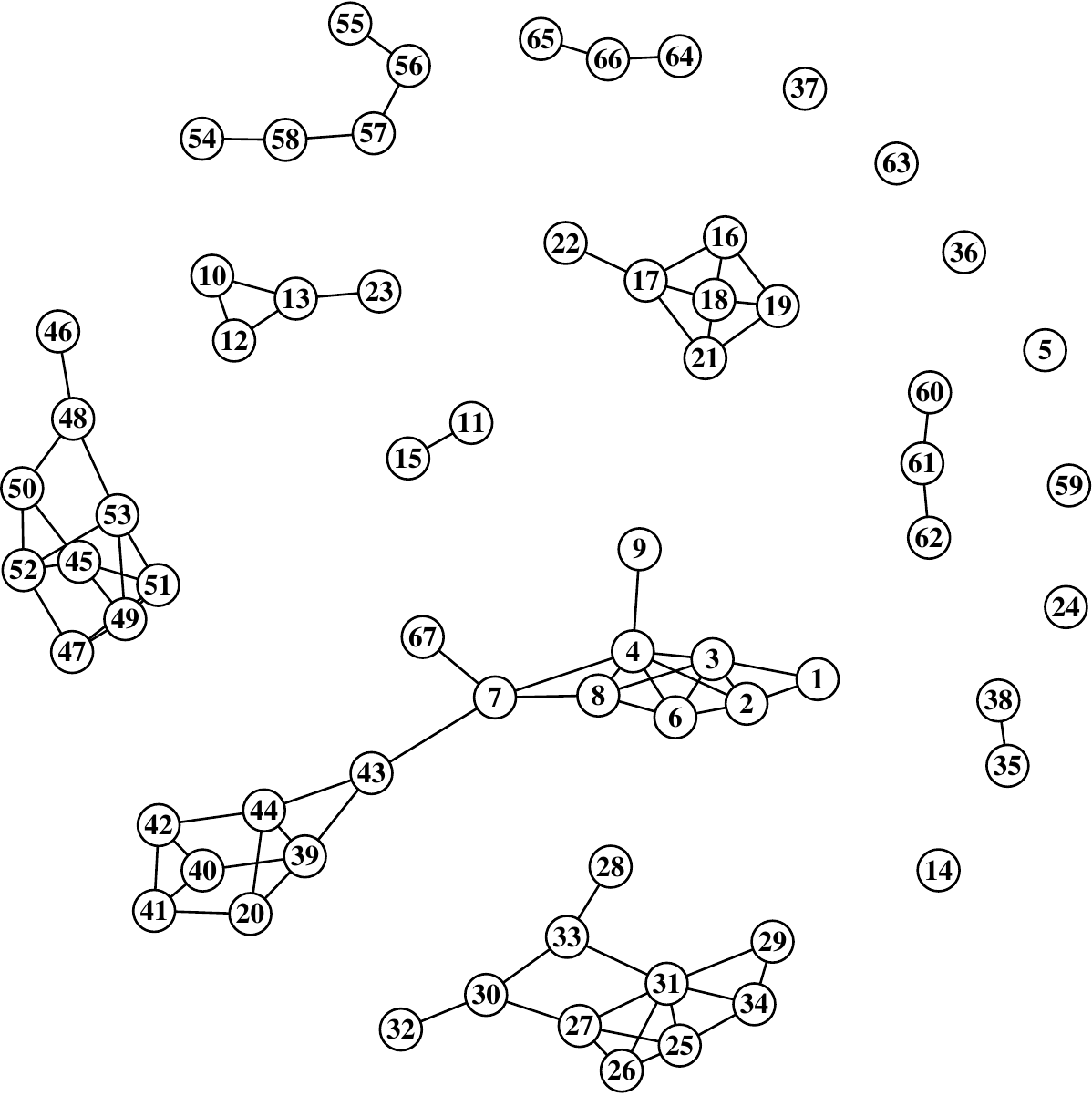}
			\caption{}
		\end{subfigure}\hfil 
		\caption{(a), (b), (c) and (d) correspond to RPPA networks for GHS-LIKE-ECM, GHS-LIKE-MCMC, GHS and PRECISE. The nodes are numbered from 1 to 67, which are proteins. The map between node numbers and protein names is given in the Supplementary Table  
			\ref{nodes_num_and_protein_names}.}
		\label{graph_est_all_methods}
	\end{figure}

	\section{Concluding Remarks}
	\label{sec:discussion}
	Our main contribution in this paper is twofold: first, we propose a fully analytical prior--penalty dual termed the \textit{graphical horseshoe-like} for inference in graphical models, and second, we provide the first ever optimality results for both the frequentist point estimate as well as the fully Bayesian posterior. Consequently, we also establish the first Bayesian optimality results for the graphical horseshoe prior of \citet{li2019graphical}. Our simulation studies clearly establish that the family of horseshoe based priors perform the best among state-of-the-art competitors across a wide range of data generating mechanisms, and suggest a potential trade-off between computational burden and statistical performance vis-\`a-vis penalized likelihood and fully Bayesian procedures. Our analysis of the RPPA data establishes the proposed approach as an effective regularizer of a gene interaction network; useful for identifying the key interactions in the disease etiology of cancer.
	
	Although we focus on the estimation of $\bOmega$, two other important aspects of network inference are edge selection and the associated uncertainty quantification. Here we use posterior credible intervals for edge selection, but it might be interesting to incorporate other methods that have been proposed for variable selection with shrinkage priors, such as 2-means \citep{bhattacharya2015dirichlet} or shrinkage factor thresholding \citep{tang2018bayesian}, with appropriate modifications. On a related note, it will be interesting to establish the Bayes risk and the oracle under $0-1$ loss and we conjecture that global-local shrinkage priors will attain such oracular risk with suitable assumptions on the prior tails and the global shrinkage parameter. Finally, it will be worth investigating whether one can extend the methods for generalized linear models, e.g. graphical models with exponential families as node-conditional distributions \citep{yang2012graphical}. It has been shown that while restricting the response distribution to natural exponential families with quadratic variance functions, shrinkage estimators enjoy certain optimality properties \citep{xie2016optimal}, and it remains to be settled whether similar properties hold true for graphical models as well.

	\section*{Appendix}
	\setcounter{section}{0}
	\renewcommand{\thesection}{A.\arabic{section}}
	\renewcommand{\thesubsection}{A.\arabic{section}.\arabic{subsection}}

	\section{Proof of Theorem~\ref{thm:main-post-conv}}
	\label{app-main-post-conv}
	We use the general theory of posterior convergence rate as outlined in Theorem 2.1 of \cite{ghosal2000}. We also refer to several auxiliary lemmas from Supplementary Section~\ref{sec:app-aux} throughout the proof. We need to show the following:
	\begin{itemize}
		\item[(i)] the prior concentration rate of Kullback--Leibler $\epsilon_n^2$-neighborhoods is at least $\exp(-cn\epsilon_n^2)$ for some constant $c > 0,$
		\item[(ii)] for a suitably chosen sieve of densities $\mathcal{P}_n$, the $\epsilon_n$-metric entropy of $\mathcal{P}_n$ is bounded by a constant multiple of $n\epsilon_n^2,$
		\item[(iii)] the probability of the complement of the above sieve is exponentially small, that is, $\Pi(\mathcal{P}_n^c) \leq \exp(-c'n\epsilon_n^2),$ for some constant $c'> 0$.
	\end{itemize}
	
	The above three parts together give the posterior convergence rate $\epsilon_n$ with respect to the Hellinger distance on the space of densities of the precision matrix. Owing to the intrinsic relationship between the Hellinger distance and the Frobenius distance for precision matrices as given by Lemma A.1 of \cite{banerjee2015bayesian}, we get the desired posterior convergence rate. 
	
	\noindent\textbf{(i) \underline{Prior concentration.}}\\
	We first define $\mathcal{B}(p_{\Ommat_0}, \eps_n)$, the $\eps_n^2$-neighborhoods of the true density in the Kullback--Leibler sense. For $K(p_1,p_2) = \int p_1\log(p_1/p_2),\; V(p_1,p_2) = \int p_1 \log^2(p_1/p_2),$ this is defined as $\mathcal{B}(p_{\Ommat_0}, \eps_n) = \{p_{\Ommat}: K(p_{\Ommat_0},p_{\Ommat}) \leq \eps_n^2, V(p_{\Ommat_0},p_{\Ommat}) \leq \eps_n^2\}.$ For $\bOmega_0 \in \mathcal{U}(\varepsilon_0,s),\,\bOmega \in \mathcal{M}_p^+(L),$ let $d_1,\ldots,d_p$ denote the eigenvalues of $\bOmega_0^{-1/2}\bOmega\bOmega^{-1/2}_0$. Then, using Lemma~\ref{lemma:KL}, we have,
	\begin{eqnarray}
		\label{eqn:KL-expressions}
		K(p_{\Ommat_0},p_{\Ommat}) &=& -\dfrac{1}{2}\sum_{i=1}^{p}\log d_i - \dfrac{1}{2}\sum_{i=1}^{n}(1 - d_i),\nonumber \\
		V(p_{\Ommat_0},p_{\Ommat}) &=& \dfrac{1}{2}\sum_{i=1}^{n}(1 - d_i)^2 + K(p_{\Ommat_0},p_{\Ommat})^2.
	\end{eqnarray}
	Note that $\sum_{i=1}^{n}(1 - d_i)^2  = \|\bmI_p - \bOmega_0^{-1/2}\bOmega\bOmega^{-1/2}_0\|_2^2,$ so that, when  $\|\bmI_p - \bOmega_0^{-1/2}\bOmega\bOmega^{-1/2}_0\|_2^2$ is small, we have, $\max_{1 \leq i \leq p}|1-d_i| < 1.$ This gives, using (\ref{eqn:KL-expressions}), 
	\begin{equation*}
		K(p_{\Ommat_0},p_{\Ommat}) = -\dfrac{1}{2}\sum_{i=1}^{p}\log d_i - \dfrac{1}{2}\sum_{i=1}^{n}(1 - d_i) \lesssim \sum_{i=1}^n(1-d_i)^2,\, V(p_{\Ommat_0},p_{\Ommat}) \lesssim \sum_{i=1}^n(1-d_i)^2.\nonumber
	\end{equation*}
	Observe that,
	\begin{eqnarray}
		\sum_{i=1}^{n}(1 - d_i)^2 &=& \|\bmI_p - \bOmega_0^{-1/2}\bOmega\bOmega^{-1/2}_0\|_2^2 = \|\bOmega_0^{-1/2}(\bOmega - \bOmega_0)\bOmega^{-1/2}_0\|_2^2 \nonumber \\
		&\leq & \|\bOmega_0^{-1}\|_2^2\|\bOmega - \bOmega_0\|_2^2 \leq \varepsilon_0^{-2}\|\bOmega - \bOmega_0\|_2^2.\nonumber
	\end{eqnarray}
	Hence, for a sufficiently small constant $c_1>0$, we have,
	\begin{equation}
		\Pi\left(\mathcal{B}(p_0, \eps_n)\right) \geq \Pi \left\lbrace \|\bOmega - \bOmega_0\|_2 \leq c_1\eps_n\right\rbrace 
		\geq \Pi \left\lbrace \|\bOmega - \bOmega_0\|_\infty \leq c_1\eps_n/p\right\rbrace. \nonumber
	\end{equation}
	The proposed prior on $\bOmega$ has a bounded spectral norm. However, such a constraint can only increase the prior concentration, since $\bOmega_0 \in \mathcal{U}(\varepsilon_0,s),\, \varepsilon_0 < L$.  Hence, we may pretend component-wise independence of the elements of $\bOmega$, so that the above expression can be simplified as products of marginal prior probabilities. We have, 
	\begin{eqnarray}
		\Pi\left(\|\bOmega - \bOmega_0\|_{\infty} \leq c_1\eps_n/p \right) &\gtrsim & (c_1\eps_n/p)^{(p+s)}\prod_{\{(i,j): \omega_{ij,0} = 0\}} \pi(|\omega_{ij}| \leq c_1\eps_n/p) \nonumber \\
		&\geq& (c_1\eps_n/p)^{(p+s)}\min_{\{(i,j): \omega_{ij,0} = 0\}}\{\pi(|\omega_{ij}| \leq c_1\eps_n/p)\}^{{\binom{p}{2}} - s }.\nonumber
	\end{eqnarray}
	Note that, from equation (\ref{eqn:lemma-prior-prob-1}) in Lemma~\ref{lemma:prior-prob}, we have, for all $1\leq i, j \leq p$, $$\{\pi(|\omega_{ij}| \leq c_1\eps_n/p)\}^{\binom{p}{2} - s } \geq \{1 - p^{-b'_1}\}^{\binom{p}{2} - s} \rightarrow 1.$$
	Thus, $\Pi\left(\|\bOmega - \bOmega_0\|_{\infty} \leq c_1\eps_n/p \right) \gtrsim (c_1\eps_n/p)^{(p+s)}.$ The prior concentration rate condition thus gives, $(p+s)(\log p + \log (1/\eps_n)) \asymp n\eps_n^2,$ so as to yield $\eps_n = n^{-1/2}(p+s)^{1/2}(\log n)^{1/2}.$

	\noindent\textbf{(ii) \underline{Choosing the sieve and controlling metric entropy.}}\\
	We now carefully choose a sieve in the space of prior densities to control its Hellinger metric entropy. Consider the sieve $\mathcal{P}_n$ such that the maximum number of elements of $\bOmega$ exceeding $\delta_n = \eps_n/p^\nu, \nu > 0$ is at most $\bar{r}_n$, and the absolute values of the entries of $\bOmega$ are at most $B$. Formally, the sieve is thus given by,
	$$\mathcal{P}_n = \{\bOmega \in \mathcal{M}_p^+(L): \sum_{j,k}\Ind(|\omega_{jk}| > \delta_n) \leq \bar{r}_n, \|\bOmega\|_\infty \leq B\},$$
	where $\delta_n = \eps_n/p^\nu$ and some sufficiently large $B > 0.$ We shall choose $B$ in such a way that the metric entropy condition is satisfied. 
	Note that, for $\bOmega_1,\bOmega_2 \in \mathcal{M}_p^+(L),\,\|\bOmega_1 - \bOmega_2\|_2^2 \leq p^2\|\bOmega_1 - \bOmega_2\|_\infty^2,$ so that, if $\|\bOmega_1 - \bOmega_2\|_\infty^2 \leq \eps_n^2/p^{2\nu}$, where $\nu$ is such that $B \leq p^{\nu - 1}$, we have, $\|\bOmega_1 - \bOmega_2\|_2^2 \leq \eps_n^2/p^{2(\nu-1)}.$ The $\eps_n/p^\nu$-metric entropy w.r.t. the $L_{\infty}$-norm is given by 
	\begin{eqnarray}
		\log \left\{\left(\dfrac{Bp^\nu}{\eps_n}\right)^p \sum_{j=1}^{\bar{r}_n}\left(\dfrac{Bp^\nu}{\eps_n}\right)^j {\binom{\binom{p}{2}}{j}} \right\}  &=& \log \left\{\left(\dfrac{Bp^\nu}{\eps_n}\right)^p\right\} + \log \left\lbrace \sum_{j=1}^{\bar{r}_n}\left(\dfrac{Bp^\nu}{\eps_n}\right)^j {\binom{\binom{p}{2}}{j}}\right\rbrace \nonumber \\	
		&\leq & \log \left\{\left(\dfrac{Bp^\nu}{\eps_n}\right)^p\right\} + \log \left\lbrace \bar{r}_n \left(\dfrac{Bp^\nu}{\eps_n}\right)^{\bar{r}_n}{\binom{p+\binom{p}{2}}{\bar{r}_n }} \right\rbrace \nonumber \\
		&\lesssim & (\bar{r}_n+p)(\log p + \log B + \log(1/\eps_n)).\nonumber
	\end{eqnarray}
	
	Choosing $\bar{r}_n \sim k_1 n\eps_n^2/\log n,\, k_1 > 0,$ and $B \sim k_2n\eps_n^2,\, k_2 > 0$, the above metric entropy is bounded by a constant multiple of $n\eps_n^2$. Since $\|\bOmega_1\|_{(2,2)} \leq p\|\bOmega_1\|_{\infty} \leq pB \leq  p^{\nu}$, and $h^2(p_1,p_2) \leq p^2\|\bOmega_1\|_{(2,2)}^2\|\bOmega_1 - \bOmega_2\|_\infty^2,$ the $\eps_n$-metric entropy with respect to the Hellinger distance is also bounded by a constant multiple of $n\eps_n^2$. Thus, the rate $\eps_n$ obtained via the prior concentration rate calculation satisfies the metric entropy condition as well. 
	
	\noindent\textbf{(iii) \underline{Controlling probability of the complement of the sieve.}}\\
	The task of controlling the probability of the complement of the sieve can further be divided into two sub-parts. Note that,
	$$\Pi(\mathcal{P}_n^c) \leq \Pi(N \geq \bar{r}_n + 1) + \Pi(\|\bOmega\|_\infty > B).$$
	We will calculate the probabilities in the right-hand side of the above display under an unconstrained case, and then take care of the truncation used in the prior for $\bOmega$, given by $\bOmega \in \mathcal{M}_p^+(L)$, by finding a suitable lower bound for the event $\{0 < L^{-1} < \|\bOmega\|_{(2,2)} < L < \infty\}.$ Let us denote the prior under the unconstrained case by $\Pi^*.$
	
	Define $N = \#\{(i,j): |\omega_{ij}| > \delta_n\}.$ Note that, $N \sim Bin(p_n^*, \nu_n)$, $p_n^* = \binom{p}{2}, \nu_n = \Pr(|\omega_{ij}| > \delta_n).$ Using results on bounding the Binomial CDF as in \cite{song2017nearly}, we have, 
	\begin{eqnarray}
	\Pi^*(N \geq \bar{r}_n+1)  &\leq &  1 - \Phi\left\{(2p_n^*H[\nu_n, \bar{r}_n/p_n^*])^{1/2} \right\} \nonumber \\
	&\leq &  (2\pi)^{-1/2}(2p_n^*H[\nu_n, \bar{r}_n/p_n^*])^{-1/2}\exp\{-p_n^*H[\nu_n, \bar{r}_n/p_n^*]\}, \nonumber
	\end{eqnarray}
 where $$p_n^*H[\nu_n, \bar{r}_n/p_n^*] = \bar{r}_n \log\left(\dfrac{\bar{r}_n}{p_n^* \nu_n}\right) + (p_n^* - \bar{r}_n) \log\left(\dfrac{p_n^* - \bar{r}_n}{p_n^* - p_n^* \nu_n}\right).$$
	It suffices to prove that $p_n^*H[\nu_n, \bar{r}_n/p_n^*] \geq O(n\eps_n^2).$ We have,
	$$p_n^*H[\nu_n, \bar{r}_n/p_n^*] \approx  \bar{r}_n \log\left(\dfrac{\bar{r}_n}{p_n \nu_n}\right) + (p_n^2 - \bar{r}_n) \log\left(\dfrac{p_n^2 - \bar{r}_n}{p_n^2 - p_n^2 \nu_n}\right).$$
	For the first term on the RHS above, we have,
	$\bar{r}_n \log\{\bar{r}_n/(p_n \nu_n)\} \geq \bar{r}_n \log \bar{r}_n + b_1' \bar{r}_n \log p_n,$
	since $\nu_n \leq p_n^{-b_1'}$ vide (\ref{eqn:lemma-prior-prob-1}) in Lemma~\ref{lemma:prior-prob}. Note that $\bar{r}_n \log p \sim \bar{r}_n \log n \asymp n\eps_n^2,$ so as to get $\bar{r}_n \log\{\bar{r}_n/(p_n \nu_n)\} \geq n\eps_n^2.$
	For the second term, we have,
	$(p_n^2 - \bar{r}_n) \log\{(p_n^2 - \bar{r}_n)/(p_n^2 - p_n^2 \nu_n)\} \asymp \bar{r}_n\left(1 - \bar{r}_n/p_n^2\right)  = o(n\eps_n^2).$
	Hence, we get, $p_n^*H[\nu_n, \bar{r}_n/p_n^*] \geq O(n\eps_n^2),$ which implies,
	\begin{equation}
		\label{eqn:sieve-compl-1}
		\Pi^*(N \geq \bar{r}_n+1) \leq \exp\{-C'n\eps_n^2\},
	\end{equation}
	for some $C' > 0.$ 
	From (\ref{eqn:lemma-prior-prob-2}) in Lemma~\ref{lemma:prior-prob}, for the choice of $B \sim k_2n\eps_n^2$ as outlined in the metric entropy condition above, we have,
	\begin{equation}
		\label{eqn:sieve-compl-2}
		\Pi^*(\|\bOmega\|_\infty > B) \leq 2p^2 \exp(-k_2n\eps_n^2).
	\end{equation}
	Combining equations~(\ref{eqn:sieve-compl-1}) and (\ref{eqn:sieve-compl-2}), we get, for a suitable constant $c_3 > 0$,
	\begin{equation}
		\label{eqn:sieve-compl-final}
		\Pi^*(N \geq \bar{r}_n+1) + \Pi^*(\|\bOmega\|_\infty > B) \lesssim \exp(-c_3n\eps_n^2).
	\end{equation}
	Combining (\ref{eqn:sieve-compl-final}) and (\ref{eqn:constrined-omega-final}), we get, for a suitable constant $c_4 > 0$,
	\begin{eqnarray}
		\Pi(\mathcal{P}_n^c)  = \dfrac{\Pi^*(\mathcal{P}_n^c)}{\Pi(\bOmega \in \mathcal{M}_p^+(L))} &\lesssim & \exp(-c_3n\eps_n^2)L^{-p}\exp(C_1p n^{-1/2}) \nonumber \\
		&=& \exp(-c_3n\eps_n^2 - p \log L + C_2p n^{-1/2}) \nonumber \\
		&\lesssim & \exp(-c_4n\eps_n^2). \nonumber
	\end{eqnarray}
	Hence, the complement of the chosen sieve has exponentially small prior probability. Thus, $\eps_n = n^{-1/2}(p+s)^{1/2}(\log n)^{1/2}$ is the posterior convergence rate and the result is established.
	
	
	\section{Proof of Theorem \ref{thm:MAP}}
	\label{app-map}
	Consider the MAP estimator of the precision matrix corresponding to the graphical horseshoe-like prior $\hat{\bOmega}^{\mathrm{MAP}}$ as outlined in Section~\ref{subsec:MAP-theory}. Define $\bDelta = (\!(\delta_{ij})\!)= \bOmega - \bOmega_0$ such that $\|\bDelta\|_2 = M\epsilon_n$, $M > 0$ is a large constant. Here, $\bOmega_0 = (\!(\omega_{ij,0})\!)$ is the true precision matrix. The true covariance matrix is $\bSigma_0 = (\!(\sigma_{ij,0})\!)$, and the natural estimator of the covariance is $\bmS = (\!(s_{ij})\!)$. Consider $Q(\bOmega)$ as defined in (\ref{eq:q-omega}). If we can show that for some small $\varepsilon > 0$, $$\mathbb{P}\left(\sup_{\|\bDelta\|_2 = M\epsilon_n} Q(\bOmega_0 + \bDelta) < Q(\bOmega_0)\right) \geq 1 - \varepsilon,$$ 
	then there exists a local maximizer $\hat{\bOmega}$ such that $\|\hat{\bOmega} - \bOmega_0\|_2 = O_P(\epsilon_n).$ We have,
	\begin{eqnarray}
		Q(\bOmega) &=& l(\bOmega) - \sum_{i < j}pen(\omega_{ij}) = \dfrac{n}{2}\log \det (\bOmega) - \dfrac{n}{2}\tr(\bmS\bOmega) -  \sum_{i < j}pen(\omega_{ij}) \nonumber \\
		&=& \dfrac{n}{2}\left\{\log \det (\bOmega) - \tr(\bmS\bOmega) -  \dfrac{2}{n}\sum_{i < j}pen(\omega_{ij}) \right\}  = \dfrac{n}{2}h(\bOmega),\; \mathrm{say}. \nonumber
	\end{eqnarray}
	Let us denote as $(2/n)\,pen(\omega_{ij})$ as $p_n(\omega_{ij})$. This gives,
	\begin{eqnarray}
		h(\bOmega_0 + \bDelta) - h(\bOmega) &=& \log \det (\bOmega_0 + \bDelta) - \tr(\bmS(\bOmega_0 + \bDelta)) - \log \det (\bOmega_0) + \tr(\bmS\bOmega_0) \nonumber \\
		&& \hspace{0.5in} - \sum_{i < j}\left\{p_n(\omega_{ij,0} + \delta_{ij}) - p_n(\omega_{ij,0})\right\}.
		\label{eqn:Step1}	
	\end{eqnarray}
	By Taylor's series expansion of logarithm of the determinant of a matrix, we have,
	\begin{eqnarray}
		&&\log \det (\bOmega_0 + \bDelta) - \log \det (\bOmega_0) \nonumber  \\
		&=&  \mathrm{tr}(\bSigma_0 \bDelta) - \mathrm{vec}(\bDelta)^T\left[ \int_{0}^{1}(1 - \nu)(\bOmega_0 + \nu \bDelta)^{-1} \otimes (\bOmega_0 + \nu \bDelta)^{-1}\, d\nu\right]\mathrm{vec}(\bDelta).\nonumber
	\end{eqnarray}
	Plugging the above in (\ref{eqn:Step1}), we have the expression for $h(\bOmega_0 + \bDelta) - h(\bOmega)$ as
	\begin{eqnarray}
		&& \tr[(\bSigma_0 - \bmS)\bDelta] - \mathrm{vec}(\bDelta)^T\left[ \int_{0}^{1}(1 - \nu)(\bOmega_0 + \nu \bDelta)^{-1} \otimes (\bOmega_0 + \nu \bDelta)^{-1}\, d\nu\right]\mathrm{vec}(\bDelta) \nonumber \\
		&& \hspace{0.5in} - \sum_{i < j}\left\{p_n(\omega_{ij,0} + \delta_{ij}) - p_n(\omega_{ij,0})\right\} \nonumber \\
		&=& \mathrm{I} + \mathrm{II} + \mathrm{III},\, \mathrm{say}. 
		\label{eqn:Step2}	
	\end{eqnarray}
	We shall now separately bound the three terms I, II, and III. For bounding I, we have,
	\begin{eqnarray}
		\left|\tr[(\bSigma_0 - \bmS)\bDelta]\right| &\leq & \left| \sum_{i \neq j}(\sigma_{ij,0} - s_{ij})\delta_{ij} \right| + \left| \sum_{i}(\sigma_{ii,0} - s_{ii})\delta_{ii}\right|.
		\label{eqn:Step3}
	\end{eqnarray}
	Using Boole's inequality and Lemma \ref{lemma:lemma1}, we have, with probability tending to one,	
	$$\max_{i \neq j}|s_{ij} - \sigma_{ij,0}| \leq C_1\left(\dfrac{\log p}{n}\right)^{1/2}.$$
	Hence, the first term in the RHS of display (\ref{eqn:Step3}) is bounded by $C_1\left(\log p/n\right)^{1/2}\|\bDelta^{-}\|_1.$ By Cauchy-Schwarz inequality and Lemma~\ref{lemma:lemma1}, we have, with probability tending to one,
	\begin{eqnarray}
		\left| \sum_{i}(\sigma_{ii,0} - s_{ii})\delta_{ii}\right| & \leq & \left\{\sum_{i}(\sigma_{ii,0} - s_{ii})^2\right\}^{1/2}\|\bDelta^+\|_2 \leq  p^{1/2}\max_{1 \leq i \leq p}|s_{ii} - \sigma_{ii,0}|\|\bDelta^+\|_2 \nonumber \\
		& \leq & C_2\left(\dfrac{p \log p}{n}\right)^{1/2}\|\bDelta^+\|_2 \leq C_2\left(\dfrac{(p+s) \log p}{n}\right)^{1/2}\|\bDelta^+\|_2. \nonumber
	\end{eqnarray}
	Thus, combining the bounds above, with probability approaching one, a bound for expression I is, 
	\begin{equation}
		\mathrm{I} \leq C_1\left(\dfrac{\log p}{n}\right)^{1/2}\|\bDelta^{-}\|_1 + C_2\left(\dfrac{(p+s) \log p}{n}\right)^{1/2}\|\bDelta^+\|_2.
		\label{eqn:bound-I}
	\end{equation}
	
	Now we proceed to find suitable bounds for expression II. Note that II is upper bounded by the negative of the minimum of $\mathrm{vec}(\bDelta)^T\left[ \int_{0}^{1}(1 - \nu)(\bOmega_0 + \nu \bDelta)^{-1} \otimes (\bOmega_0 + \nu \bDelta)^{-1}\, d\nu\right]\mathrm{vec}(\bDelta).$ Using the result that $\min_{\|\bx\|_2 = 1}\bx^T \bmA \bx = \mathrm{eig}_1(\bmA)$, we have,
	\begin{eqnarray}
		&& \min\left\{\mathrm{vec}(\bDelta)^T\left[ \int_{0}^{1}(1 - \nu)(\bOmega_0 + \nu \bDelta)^{-1} \otimes (\bOmega_0 + \nu \bDelta)^{-1}\, d\nu\right]\mathrm{vec}(\bDelta)\right\} \nonumber \\
		&=& \|\bDelta\|_2^2 \, \mathrm{eig}_1\left[ \int_{0}^{1}(1 - \nu)(\bOmega_0 + \nu \bDelta)^{-1} \otimes (\bOmega_0 + \nu \bDelta)^{-1}\, d\nu\right] \nonumber \\
		& \geq & \|\bDelta\|_2^2\, \int_{0}^{1}(1 - \nu)\mathrm{eig}_1^2(\bOmega_0 + \nu \bDelta)^{-1}d\nu \geq \dfrac{1}{2}\|\bDelta\|_2^2 \min_{0 \leq \nu \leq 1}\mathrm{eig}_1^2(\bOmega_0 + \nu \bDelta)^{-1} \nonumber \\
		& \geq & \dfrac{1}{2}\min\left\{\mathrm{eig}_1^2(\bOmega_0 + \bDelta)^{-1}: \|\bDelta\|_2 \leq M\epsilon_n \right\}\nonumber.
	\end{eqnarray}
	Note that, $\mathrm{eig}_1^2(\bOmega_0 + \bDelta)^{-1} = \mathrm{eig}_p^{-2}(\bOmega_0 + \bDelta) \geq (\|\bOmega_0\|_{(2,2)} + \|\bDelta\|_{(2,2)})^{-2} \geq \varepsilon_0^2/2,$ with probability tending to one. The last inequality follows from the fact that $\|\bDelta\|_{(2,2)} \leq \|\bDelta\|_2 = o(1).$ Hence, with probability tending to one, we have the bound for expression II as
	\begin{equation}
		\mathrm{II} \leq - \dfrac{1}{4}\varepsilon_0^2 \|\bDelta\|_2^2.
		\label{eqn:bound-II}
	\end{equation}
	
	Finally, we proceed to find suitable bounds for expression III. Let us denote the set $\mathcal{S} = \{(i,j): \omega_{ij,0} = 0, i < j\}$. This set comprises of the indices in the uppper triangle of the true precision matrix that are exactly equal to zero. The complement of $\mathcal{S}$ consists of the indices with non-zero entries in the uppper triangle of the same. We can partition expression III (without the negative sign) as 
	\begin{align}
		\sum_{i < j}\left\{p_n(\omega_{ij,0} + \delta_{ij}) - p_n(\omega_{ij,0})\right\} &= \sum_{(i,j) \in \mathcal{S}}\left\{p_n(\omega_{ij,0} + \delta_{ij}) - p_n(\omega_{ij,0})\right\} + \sum_{(i,j) \in \mathcal{S}^c}\left\{p_n(\omega_{ij,0} + \delta_{ij}) - p_n(\omega_{ij,0})\right\} \nonumber\\
		&= \sum_{(i,j) \in \mathcal{S}}\left\{p_n(\delta_{ij}) - p_n(0)\right\}  + \sum_{(i,j) \in \mathcal{S}^c}\left\{p_n(\omega_{ij,0} + \delta_{ij}) - p_n(\omega_{ij,0})\right\} \nonumber \\
		& >  \dfrac{M'}{n} + \sum_{(i,j) \in \mathcal{S}^c}\left\{p_n(\omega_{ij,0} + \delta_{ij}) - p_n(\omega_{ij,0})\right\}. \nonumber
	\end{align}
	The last inequality follows from the fact that $pen(\theta) \rightarrow -\infty$ as $|\theta| \rightarrow 0$, and hence the first term in the above expression is larger than $M'/n$ for some large $M' > 0$. This implies: 
	$$\sum_{i < j}\left\{p_n(\omega_{ij,0} + \delta_{ij}) - p_n(\omega_{ij,0})\right\} > \sum_{(i,j) \in \mathcal{S}^c}\left\{p_n(\omega_{ij,0} + \delta_{ij}) - p_n(\omega_{ij,0})\right\}.$$
	By Taylor's series expansion of $p_n(\omega_{ij} + \delta_{ij})$ around $\omega_{ij,0} (\neq 0)$, we have,
	$$p_n(\omega_{ij} + \delta_{ij}) - p_n(\omega_{ij,0}) = \delta_{ij}p_n'(\omega_{ij,0}) + \dfrac{1}{2}\delta_{ij}^2 p_n''(\omega_{ij,0})(1 + o(1)).$$
	Since $-x \leq |x|$, we can write,
	\begin{eqnarray}	
		&& -\sum_{(i,j) \in \mathcal{S}^c}\left\{p_n(\omega_{ij,0} + \delta_{ij}) - p_n(\omega_{ij,0})\right\} \nonumber \\
		& \leq & \max\left\{|p_n'(\omega_{ij,0})| \right\} \sum_{(i,j) \in \mathcal{S}^c}|\delta_{ij}| + \dfrac{1}{2}\max\left\{|p_n''(\omega_{ij,0})| \right\} \sum_{(i,j) \in \mathcal{S}^c}\delta_{ij}^2(1 + o(1)) \nonumber \\
		& \leq & \max\left\{|p_n'(\omega_{ij,0})| \right\} \|\bDelta^{-1}\|_1 + \dfrac{1}{2}\max\left\{|p_n''(\omega_{ij,0})| \right\} \|\bDelta^-\|_2^2(1 + o(1)) \nonumber \\
		& \leq & s^{1/2}\max\left\{|p_n'(\omega_{ij,0})| \right\} \|\bDelta\|_2 + \dfrac{1}{2}\max\left\{|p_n''(\omega_{ij,0})| \right\} \|\bDelta\|_2^2(1 + o(1)).
		\label{eqn:Step4}
	\end{eqnarray}
	Now, note that, 
	$$|p_n'(\omega_{ij,0})| = \dfrac{2}{n}|pen'(\omega_{ij,0})| = \dfrac{2}{n}\dfrac{2a/|\omega_{ij,0}|^3}{(1+ \dfrac{a}{\omega_{ij,0}^2})\log\left(1 + \dfrac{a}{\omega_{ij,0}^2}\right)}.$$
	Since $(1+x)\log(1+x) > x$ for $x > -1, x \neq 0$, we have,
	$|p_n'(\omega_{ij,0})| <  4/(n|\omega_{ij,0}|).$
	We now arrive at a suitable bound for the double derivative of the penalty. Note that, for $\theta \neq 0$,
	\begin{eqnarray}
		pen''(\theta) &=& -\dfrac{2a\left\{(a + 3\theta^2)\log\left(1 + \dfrac{a}{\theta^2}\right) - 2a\right \}}{\theta^6\left(1 + \dfrac{a}{\theta^2}\right)^2\log^2\left(1 + \dfrac{a}{\theta^2}\right)} \nonumber \\
		&\leq&	-\dfrac{2a\left\{(a + 3\theta_0^2)\log\left(1 + \dfrac{a}{\theta_0^2}\right) - 2a\right \}}{\theta^6\left(1 + \dfrac{a}{\theta^2}\right)^2\log^2\left(1 + \dfrac{a}{\theta^2}\right)}, \nonumber
	\end{eqnarray}
	where $\theta_0 = \arg\max_\theta  \left\{-(a+3\theta^2)\log(1 + {a}/{\theta^2}) \right\} = (ak)^{1/2}, \; k = \{\exp(z_0)-1\}^{-1},\, z_0 \approx 1.0356.$ Hence,
	\begin{eqnarray}
		|pen''(\theta)| & \leq & \dfrac{2a\left|(a + 3\theta_0^2)\log\left(1 + \dfrac{a}{\theta_0^2}\right) - 2a\right |}{\theta^6\left(1 + \dfrac{a}{\theta^2}\right)^2\log^2\left(1 + \dfrac{a}{\theta^2}\right)} \nonumber \\
		&\leq & \dfrac{2a\left|(a + 3\theta_0^2)\log\left(1 + \dfrac{a}{\theta_0^2}\right) - 2a\right |}{\theta^6 \dfrac{a^2}{\theta^4}} = \dfrac{2\left|(a + 3\theta_0^2)\log\left(1 + \dfrac{a}{\theta_0^2}\right) - 2a\right |}{a\theta^2} \nonumber \\
		& = & \dfrac{2\left|(a + 3ak)\log\left(1 + \dfrac{a}{ak}\right) - 2a\right |}{a\theta^2} \approx \dfrac{C_3}{\theta^2}\nonumber,
	\end{eqnarray}
	where $C_3 > 0$ is a constant not depending on $n$ or $a$. This gives,
	$$|p_n''(\omega_{ij,0})| = \dfrac{2}{n}|pen''(\omega_{ij,0})| < \dfrac{2C_3}{n\min_{(i,j) \in \mathcal{S}^c}\omega_{ij,0}^2}.$$
	Thus, expression III can be bounded as follows:
	\begin{equation}
		\mathrm{III} \leq s^{1/2}\|\bDelta\|_2\dfrac{4}{n\min_{(i,j) \in \mathcal{S}^c}|\omega_{ij,0}|} + \dfrac{C_3}{n\min_{(i,j) \in \mathcal{S}^c}\omega_{ij,0}^2}\|\bDelta\|_2^2(1+o(1)).
		\label{eqn:bound-III}
	\end{equation}
	Combining Equations (\ref{eqn:bound-I}), (\ref{eqn:bound-II}), and (\ref{eqn:bound-III}), we have, with probability tending to one,
	\begin{eqnarray}
		&& Q(\bOmega_0 + \bDelta) - Q(\bOmega_0) \nonumber \\
		& \leq &  C_1\left(\dfrac{\log p}{n}\right)^{1/2}\|\bDelta^{-}\|_1 + C_2\left(\dfrac{(p+s) \log p}{n}\right)^{1/2}\|\bDelta^+\|_2 - \dfrac{1}{4}\varepsilon_0^2 \|\bDelta\|_2^2 \nonumber \\
		&& + \,s^{1/2}\|\bDelta\|_2\dfrac{4}{n\min_{(i,j) \in \mathcal{S}^c}|\omega_{ij,0}|} + \dfrac{C_3}{n\min_{(i,j) \in \mathcal{S}^c}\omega_{ij,0}^2}\|\bDelta\|_2^2(1+o(1)) \nonumber \\
		& \leq &  C_1\left(\dfrac{(p+s)\log p}{n}\right)^{1/2}\|\bDelta\|_2 + C_2\left(\dfrac{(p+s)\log p}{n}\right)^{1/2}\|\bDelta\|_2 - \dfrac{1}{4}\varepsilon_0^2 \|\bDelta\|_2^2 \nonumber \\
		&& + \,s^{1/2}\|\bDelta\|_2\dfrac{4}{n\min_{(i,j) \in \mathcal{S}^c}|\omega_{ij,0}|} + \dfrac{C_3}{n\min_{(i,j) \in \mathcal{S}^c}\omega_{ij,0}^2}\|\bDelta\|_2^2 (1+o(1)) \nonumber \\
		&\leq & \|\bDelta\|_2^2\left\{ C_1\left(\dfrac{(p+s)\log p}{n}\right)^{1/2}\|\bDelta\|_2^{-1} + C_2\left(\dfrac{(p+s)\log p}{n}\right)^{1/2}\|\bDelta\|_2^{-1} - \dfrac{1}{4}\varepsilon_0^2 \right. \nonumber \\
		&& \left. + \,(p+s)^{1/2}\|\bDelta\|_2^{-1} \dfrac{4}{n\min_{(i,j) \in \mathcal{S}^c}|\omega_{ij,0}|} + \dfrac{C_3}{n\min_{(i,j) \in \mathcal{S}^c}\omega_{ij,0}^2}(1+o(1)) \right\}  \nonumber \\
		&=& \|\bDelta\|_2^2\left\{\dfrac{C_1}{M} + \dfrac{C_2}{M} - \dfrac{1}{4}\varepsilon_0^2 + \dfrac{4}{(n \log p)^{1/2}\min_{(i,j) \in \mathcal{S}^c}|\omega_{ij,0}|} + \dfrac{C_3}{n\min_{(i,j) \in \mathcal{S}^c}\omega_{ij,0}^2}(1+o(1))\right\} < 0, \nonumber
		\label{eqn:finalStep}
	\end{eqnarray}
	for $M$ sufficiently large, and owing to the fact that the last two terms inside the bracket in the above display are $o(1)$ as $\min_{(i,j) \in \mathcal{S}^c}|\omega_{ij,0}|$ are bounded away from zero. This completes the proof.
	

\bigskip
\begin{center}
{\large\bf SUPPLEMENTARY MATERIAL}
\end{center}
Supplementary Material available online includes a summary of notations, additional
details on the proposed prior, auxiliary lemmas used in the
main theorems, and additional results on simulated and proteomics data. It also contains a computer code archive along with an instructional README file.


\bibliographystyle{apalike}
\bibliography{ref, hs-review}

\clearpage\pagebreak\newpage
\begin{center}
	{\LARGE{\bf Supplementary Material to\\
	{\it Precision Matrix Estimation under the Horseshoe-like Prior--Penalty Dual}\\
	{ by \vspace{2mm}}\\
	{\it K. Sagar, S. Banerjee, J. Datta and A. Bhadra}
	}
	}
\end{center}

\setcounter{equation}{0}
\setcounter{page}{1}
\setcounter{table}{0}
\setcounter{section}{0}
\setcounter{subsection}{0}
\setcounter{figure}{0}
\renewcommand{\theequation}{S.\arabic{equation}}
\renewcommand{\thesection}{S.\arabic{section}}
\renewcommand{\thesubsection}{S.\arabic{subsection}}
\renewcommand{\thepage}{S.\arabic{page}}
\renewcommand{\thetable}{S.\arabic{table}}
\renewcommand{\thefigure}{S.\arabic{figure}}

	
	\section{Notations and Preliminaries}
	\label{sec:notations}
	For positive real-valued sequences $\{a_n\}$ and $\{b_n\}$, $a_n = O(b_n)$ means that $a_n/b_n$ is bounded, and $a_n = o(b_n)$ means that $a_n/b_n \rightarrow 0$ as $n \rightarrow \infty$;  $a_n \lesssim b_n$ implies that $a_n = O(b_n)$, and $a_n \asymp b_n$ means that both $a_n = O(b_n)$ and $b_n = O(a_n)$ hold. For a sequence of random variables $\{X_n\}$, $X_n = O_P(\eps_n)$ means that $\P(|X_n| \leq M\eps_n) \rightarrow 1$ for some constant $M > 0.$
	
	Vectors are represented in bold lowercase English or Greek letters, with corresponding components denoted by non-bold letters, for example, $\bx = (x_1,\ldots,x_p)^T.$ For a vector $\bx \in \RR^p$, the $L_r$-norm, for $0 < r < \infty$,  is defined as $\|\bx\|_r = \left(\sum_{i=1}^p |x_j|^{r}\right)^{1/r}$, and the $L_{\infty}$-norm is defined as $\|\bx\|_{\infty} = \max_{1\leq j \leq p}|x_j|.$ The zero-vector is denoted by $\zerovec$. Matrices are represented in bold uppercase English or Greek letters, for example, $\bmA = (\!(a_{ij})\!),$ where $a_{ij}$ denotes the $(i,j)$th entry of $\bmA$. We denote the identity matrix by $\bmI_p$. For a symmetric matrix $\bmA$, $\eig_1(\bmA) \leq \ldots \leq \eig_p(\bmA)$ denote the ordered eigenvalues of $\bmA$, and its trace and determinant are denoted by $\tr(\bmA)$ and $\det \bmA$ respectively. The $L_r$ and $L_{\infty}$-norms on $p \times p$ matrices are respectively defined as $\|\bmA\|_r = \left(\sum_{i=1}^p \sum_{j=1}^p |a_{ij}|^r \right)^{1/r},\, 0 < r < \infty$, and $\|\bmA\|_{\infty} = \max_{1\leq i, j \leq p}|a_{ij}|$. In particular, the $L_2$-norm, or the Frobenius norm can be expressed as $\|\bmA\|_2 = \left\{\tr(\bmA^T \bmA)\right\}^{1/2}$. The $L_r$-operator norm of $\bmA$ is defined as $\|\bmA\|_{(r,r)} = \sup\{\|\bmA\bx\|_r:\|\bx\|_r = 1\}.$ This gives the $L_1$-operator norm as $\|\bmA\|_{(1,1)} = \max_{1\leq j \leq p}\sum_{i=1}^{p}|a_{ij}|$, and the $L_2$-operator norm as $\|\bmA\|_{(2,2)} = [\max_{1 \leq i \leq p}\{\eig_i(\bmA^T \bmA)\}]^{1/2}$, so that, for symmetric matrices, $\|\bmA\|_{(2,2)} = \max_{1 \leq i \leq p}|\eig_i(\bmA)|$. For a symmetric $p$-dimensional matrix $\bmA$, we have, $\|\bmA\|_\infty \leq \|\bmA\|_{(2,2)} \leq \|\bmA\|_2 \leq p\|\bmA\|_\infty$, and $\|\bmA\|_{(2,2)} \leq \|\bmA\|_{(1,1)}$. For a positive definite matrix $\bmA$, $\bmA^{1/2}$ denotes its unique positive square root. The diagonal matrix with the same diagonal as a matrix $\bmA$ is denoted by $\bmA^+$, and $\bmA^-$ denotes the matrix $\bmA - \bmA^+$. The linear space of $p \times p$ symmetric matrices is denoted by $\mathcal{M}_p,$ and $\mathcal{M}_p^+ \subset \mathcal{M}_p$ is the cone of symmetric positive definite matrices of dimension $p\times p$. 
	
	The indicator function is denoted by $\Ind.$ We denote the cardinality of a finite set $\mathcal{S}$ by $\#\mathcal{S}$. The Hellinger distance between two probability densities $f$ and $g$ is defined as $h(f,g) = \|f^{1/2} - g^{1/2}\|_2.$
	
	\section{The marginal graphical horseshoe-like prior and implications for estimation algorithms}
	\label{supp-ghs-prior-details}
	
	The graphical horseshoe-like prior on the individual off-diagonal elements $\omega_{ij}$ has a nice Gaussian scale mixture representation as outlined in Section~\ref{prior_specification}. However, the marginal prior on these elements are not horseshoe-like, owing to the positive definite constraint on the precision matrix $\bOmega$. In this section, we argue that the hierarchical representation based on the scale-mixtures induces the proposed marginal prior on $\bOmega$ and all the related marginal and conditional distributions are proper. Alongside this, we also argue that the intractable normalizing constant in the marginal prior of $\bOmega$ does not affect the conditional expectation calculations for executing the expectation conditional maximization steps in our computations.
	
	The marginal prior on $\bOmega$ given the global scale parameter $a$ can be written as,
	\begin{equation}
		\label{supp:marginal-prior-Omega}
		\pi(\bOmega \mid a) = C(a)^{-1}\prod_{i < j}\pi(\omega_{ij}\mid a)\Ind_{\mathcal{M}_p^+}(\bOmega),
	\end{equation}
	where $C(a)$ is the normalizing constant depending on $a$. Using the Gaussian scale-mixture representation, we have a hierarchical representation of the above prior as,
	\begin{equation}
		\label{supp:marginal-prior-1}
		\pi(\bOmega \mid \nu, a) = C(\nu,a)^{-1}\prod_{i < j} \pi(\omega_{ij} \mid \nu_{i,j},a) \Ind_{\mathcal{M}_p^+}(\bOmega),
	\end{equation}
	where $C(\nu,a)$ is an intractable constant depending on $\nu$ and $a$. The prior on $\nu$ is,
	\begin{equation}
		\label{supp:marginal-prior-2}
		\pi(\nu) \propto C(\nu,a) \prod_{i < j}\pi(\nu_{ij}) = C_2(a)^{-1}C(\nu,a)\prod_{i < j}\pi(\nu_{ij}),
	\end{equation}
	where $C_2(a)$ is a constant such that
	\begin{equation*}
		C_2(a) = \int C(\nu,a)\prod_{i<j} \pi(\nu_{ij})\,d\nu.
	\end{equation*}
	The constant $C(a)$ in \eqref{supp:marginal-prior-Omega} is finite because,
	\begin{equation*}
		C(a) = \int \prod_{i < j}\pi(\omega_{ij}\mid a)\Ind_{\mathcal{M}_p^+}(\bOmega)\, d(\omega_{ij})_{i \leq j} 
		< 2K \int \prod_{i < j}\pi(\omega_{ij}\mid a)\, d(\omega_{ij})_{i < j} < \infty,
	\end{equation*}
	where $K < \infty$ is such that $|\omega_{ii}| < K, (i = 1,\ldots,p)$, since $\bOmega$ is restricted to be positive definite and hence the diagonal elements are finite.
	Also, 
	\begin{equation*}
		C(\nu,a) = \int \prod_{i < j} \pi(\omega_{ij} \mid \nu_{i,j},a) \Ind_{\mathcal{M}_p^+}(\bOmega)\,d(\omega_{ij})_{i \leq j}
		< 2K \int \prod_{i < j}\pi(\omega_{ij} \mid \nu_{i,j},a)\, d(\omega_{ij})_{i < j} < \infty. 
	\end{equation*}
	The induced marginal prior on $\bOmega$ based on the hierarchical representation as in (\ref{supp:marginal-prior-1}) and (\ref{supp:marginal-prior-2}) is,
	\begin{equation*}
		\pi^*(\bOmega \mid a) = C_2(a)^{-1}\prod_{i < j}\pi(\omega_{ij}\mid a)\Ind_{\mathcal{M}_p^+}(\bOmega).
	\end{equation*}
	Since $\int \pi(\bOmega \mid a)\,d\bOmega = \int \pi^*(\bOmega \mid a)\, d\bOmega = 1$, it immediately implies that $C(a) = C_2(a).$ Thus, the intractable constant $C(\nu,a)$ in (\ref{supp:marginal-prior-1}) and (\ref{supp:marginal-prior-2}) cancels out in the hierarchical representation, so as to arrive at the induced marginal prior (\ref{supp:marginal-prior-Omega}). The above results also establish that the priors $\pi(\bOmega \mid a), \pi(\bOmega \mid \nu, a),$ and $\pi(\nu)$ are proper.
	
	We now show that it suffices to consider the component-wise scale-mixture representation of the horseshoe-like prior to find the conditional expectation of the latent parameters $\nu_{ij}$ in the expectation step (see equation~\ref{e_step}) of the expectation conditional maximization algorithm. 
	The conditional distribution of $\nu$ given $\bOmega$ and $a$ can be written as,  
	\begin{eqnarray}
		\pi(\nu \mid \bOmega, a) &=& \dfrac{\pi(\bOmega, \nu \mid a)}{\pi(\bOmega \mid a)}  = \dfrac{\pi(\bOmega \mid  \nu, a)\pi(\nu)}{\pi(\bOmega \mid a)}\nonumber \\
		& = & \dfrac{\prod_{i<j} \pi(\omega_{ij} \mid \nu_{ij},a) \prod_{i < j}\pi(\nu_{ij})\Ind_{\mathcal{M}_p^+}(\bOmega)}{\prod_{i<j} \pi(\omega_{ij}\mid a) \Ind_{\mathcal{M}_p^+}(\bOmega)}. \nonumber
	\end{eqnarray}
	This gives,
	\begin{equation*}
		\pi(\nu_{ij} \mid \bOmega, a) = \dfrac{\pi(\omega_{ij} \mid \nu_{ij},a) \pi(\nu_{ij})}{\pi(\omega_{ij}\mid a)} \Ind_{\mathcal{M}_p^+}(\bOmega).
	\end{equation*}
	Thus, the expectation step (\ref{e_step}) holds given that the conditional maximization step produces positive definite estimates of $\bOmega$ in each iteration.

	\section{Auxiliary Lemmas}\label{sec:app-aux}
	\begin{lemma}
		\label{lemma:KL}
		Let $p_k$ denote the density of a $\mathcal{N}_d(\zerovec,\bSigma_k)$ random variable, $k = 1,2.$ Denote the corresponding precision matrices by $\bOmega_k = \bSigma_k^{-1},k=1,2.$ Then,
		\begin{eqnarray}
			\mathbb{E}_{p_1}\left\lbrace \log \dfrac{p_1}{p_2}(\bX)\right \rbrace  &=& \dfrac{1}{2}\left\lbrace \log \det \bOmega_1 - \log \det \bOmega_2 + \tr(\bOmega^{-1}_1\bOmega_2) - d\right\rbrace, \nonumber \\
			\Var_{p_1}\left\lbrace \log \dfrac{p_1}{p_2}(\bX)\right \rbrace &=& \dfrac{1}{2}\, \tr\{(\bOmega_1^{-1/2}\bOmega_2\bOmega_1^{-1/2} - \bmI_d)^2\}. \nonumber
		\end{eqnarray}
	\end{lemma}
	
	\begin{proof}
		Let us define $\bmA = \bOmega_1^{-1/2}\bOmega_2\bOmega^{-1/2}.$ Note that, for a random variable $\bm{Z} \sim \mathcal{N}_d(\zerovec,\bSigma)$, we have,
		$$\mathbb{E}(\Zmat^T\Lambdamat \Zmat) = \tr(\Lambdamat \bSigma),\; \Var(\Zmat^T \Lambdamat \Zmat) = 2\,\tr(\Lambdamat \bSigma \Lambdamat \bSigma).$$
		Then, for $\bX \sim \mathcal{N}_d(\zerovec,\bSigma_1)$,
		\begin{eqnarray}
			\mathbb{E}_{p_1}\left\lbrace \log \dfrac{p_1}{p_2}(\bX)\right \rbrace  &=& \dfrac{1}{2} \left\lbrace \log \det \bOmega_1 - \log \det \bOmega_2 + \mathbb{E}_{p_1}(\bX^T(\bOmega_2 - \bOmega_1)\bX)\right\rbrace \nonumber \\
			&=& \dfrac{1}{2} \left\lbrace \log \det \bOmega_1 - \log \det \bOmega_2 + \tr[(\bOmega_2 - \bOmega_1)\bSigma_1] \right\rbrace \nonumber \\
			&=& \dfrac{1}{2}\left\lbrace \log \det \bOmega_1 - \log \det \bOmega_2 + \tr(\bOmega^{-1}_1\bOmega_2) - d\right\rbrace. \nonumber
		\end{eqnarray}
		Also, 
		\begin{eqnarray}
			\Var_{p_1}\left\lbrace \log \dfrac{p_1}{p_2}(\bX)\right \rbrace &=& \mathbb{E}_{p_1}\left[ \log \dfrac{p_1}{p_2}(\bX) - \mathbb{E}_{p_1}\left\lbrace \log \dfrac{p_1}{p_2}(\bX)\right \rbrace \right]^2 \nonumber \\
			&=& \dfrac{1}{4}\mathbb{E}_{p_1}\{\bX^T(\bOmega_2 - \bOmega_1)\bX - \mathbb{E}_{p_1}(\bX^T(\bOmega_2 - \bOmega_1)\bX) \}^2 \nonumber \\
			&=& \dfrac{1}{4}\Var_{p_1}\{\bX^T(\bOmega_2 - \bOmega_1)\bX\} \nonumber \\
			&=& \dfrac{1}{4}2\,\tr\{(\bOmega_2-\bOmega_1)\bOmega_1^{-1}(\bOmega_2-\bOmega_1)\bOmega_1^{-1}\} \nonumber \\
			&=& \dfrac{1}{2}\,\tr\{(\bOmega_1^{-1/2}\bOmega_2\bOmega_1^{-1/2} - \bmI_d)^2\}. \nonumber
		\end{eqnarray}
	\end{proof}
	
	\begin{lemma}
		\label{lemma:prior-prob}
		Consider the horseshoe-like prior $\pi(\theta \mid a).$ Then, for the global shrinkage parameter $a$ satisfying the condition $a < p^{-2b_1}/n$ for some constant $b_1 > 0, \nu > 0$, we have,
		\begin{equation}
			\label{eqn:lemma-prior-prob-1}
			1 - \int_{-\eps_n/p^\nu}^{\eps_n/p^\nu}\pi(\theta \mid a)\,d\theta \leq p^{-b'_1},
		\end{equation}
		for some constants $\nu, b'_1 > 0$.
		Additionally, for some sufficiently large constant $B \sim b_2n\eps_n^2$, if the global scale parameter satisfies the condition $a/B^2 < p^{-2b_3}$ for some constant $b_3 > 0$, we have,
		\begin{equation}
			\label{eqn:lemma-prior-prob-2}
			-\log\left( \int_{|\theta| > B}\pi(\theta \mid a)\,d\theta \right) \gtrsim B.
		\end{equation}
	\end{lemma}
	
	\begin{proof}
		We have,
		\begin{eqnarray}
			1 - \int_{-\eps_n/p^\nu}^{\eps_n/p^\nu}\pi(\theta \mid a)\,d\theta &=& \int_{|\theta| > \eps_n/p^\nu}\pi(\theta \mid a)\,d\theta \nonumber \\
			&=& \int_{|\theta| > \eps_n/p^\nu} \dfrac{1}{2\pi a^{1/2}}\log \left(1 + \dfrac{a}{\theta^2}\right)\,d\theta \nonumber \\
			&=& 2 \int_{\eps_n/p^\nu}^{\infty} \dfrac{1}{2\pi a^{1/2} }\log \left(1 + \dfrac{a}{\theta^2}\right)\,d\theta \nonumber \\
			& \leq & \int_{\eps_n/p^\nu}^{\infty} \dfrac{1}{\pi a^{1/2} } \dfrac{a}{\theta^2}\,d\theta = \dfrac{2}{\pi}\dfrac{a^{1/2}p^\nu}{\eps_n}. \nonumber
		\end{eqnarray}
		Note that, for $a^{1/2} < n^{-1/2}p^{-b_1}$, the right hand side of the display above is bounded by $p^{-b_1'}$, for $0< b_1' \leq b_1 - \nu.$ This proves the first part of the lemma. For the second part, note that,
		\begin{equation*}
			\int_{|\theta| > B}\pi(\theta \mid a)\,d\theta \leq  \dfrac{2}{\pi}\dfrac{a^{1/2}}{B}.
		\end{equation*}
		Hence, for the condition $a^{1/2}/B < p^{-b_3},$ we have,
		\begin{equation*}
			\int_{|\theta| > B}\pi(\theta \mid a)\,d\theta \lesssim  p^{-b_3} = \exp(-b_3\log p) \lesssim  \exp(-b_2n\eps_n^2),
		\end{equation*}
		which implies that, for $B \sim b_2n\eps_n^2,$
		$$-\log\left( \int_{|\theta| > B}\pi(\theta \mid a)\,d\theta \right) \gtrsim B.$$ 
		This completes the proof.
	\end{proof}
	
	\begin{corollary}
		\label{corr:GHS-prior-prob}
		The above lemma holds true under the same conditions on the global shrinkgae parameter for the horseshoe prior as well. 
	\end{corollary}
	
	\begin{proof}
		Note that the prior density of the horseshoe prior satisfies
		\begin{equation}
			p_{HS}(\theta \mid a) < \dfrac{2}{a^{1/2}(2\pi)^{3/2}}\log\left(1 + \dfrac{2a}{\theta^2}\right),
		\end{equation}
		which implies that, retracing the steps in the proof of Lemma~\ref{lemma:prior-prob} above,
		\begin{equation}
			\int_{|\theta| > t}p_{HS}(\theta \mid a)\,d\theta \lesssim \dfrac{a^{1/2} }{t}. 
		\end{equation}
		The result thus follows immediately.
	\end{proof}
	
	We now present the Gershgorin Circle Theorem \citep{brualdi1994regions}, that will be required in the proof of our main result on posterior convergence rate. The actual theorem holds for complex matrices, but we only need the result for real matrices.
	
	\begin{theorem}[Gershgorin Circle Theorem for real matrices]
		\label{thm:Gershgorin}
		Let $\bmA = (\!(a_{ij})\!)$ be a $p$-dimensional real-valued matrix with real eigenvalues. Define $R_i = \sum_{j\neq i}|a_{ij}|,\, i = 1,\ldots,p,$ the row sums of the absolute entries of $\bmA$ excluding the diagonal element. Then, each eigenvalue of $\bmA$ is in at least one of the disks 
		$$\mathcal{D}_i(\bmA) = \{z: | z - a_{ii}| \leq R_i\},\; 1 \leq i \leq p.$$
		Equivalently, the $p$ eigenvalues of $\bmA$ are contained in the region in the real plane determined by
		$$\mathcal{D}(\bmA) = \cup_{i=1}^{p}\mathcal{D}_i(\bmA).$$
	\end{theorem}
	
	\begin{proof}
		The eigenvalue equation for $\bmA$ is given by $\bmA \bx = \lambda \bx$, where $\lambda$ is an eigenvalue of $\bmA$ and $\bx = (x_1,\ldots,x_p)^T \in \RR^p$ is the corresponding non-zero eigenvector. Let us consider $1 \leq m \leq p$ such that $|x_m| = \|\bx\|_\infty.$ Then, the above eigenvalue equation implies that, $\sum_{j=1}^{p}a_{mj}x_j = \lambda x_m.$ Rearranging the terms, we get,
		$\sum_{j \neq m}a_{mj}x_j = (\lambda - a_{mm})x_m,$ which implies that,
		$$|\lambda - a_{mm}||x_m| = \left| \sum_{j \neq m}a_{mj}x_j \right| \leq \sum_{j \neq m}|a_{mj}||x_j| \leq |x_m| \sum_{j \neq m}|a_{mj}|.$$
		Hence, for any eigenvalue $\lambda$ of $\bmA$, we have, $|\lambda - a_{mm}| \leq \sum_{j \neq m}|a_{mj}|.$ Thus, each of the $p$ eigenvalues of $\bmA$ must lie in at least one of the disks $\mathcal{D}_i(\bmA)$ as defined in the theorem above. This completes the proof. 
	\end{proof}
	
	\begin{lemma}
		\label{lemma:prior-constraint}
		For the graphical horseshoe-like prior \eqref{prior}, under the assumption that the global scale parameter satisfies the condition $a^{1/2} < L^{-1}n^{-1/2}p^{-(2+u)},\, u > 0,$ the prior probability owing to the constraint $\bOmega \in \mathcal{M}_p^+(L)$ has the lower bound
		\begin{equation}
			\label{eqn:constrined-omega-final}
			\Pi(\bOmega \in \mathcal{M}_p^+(L)) \gtrsim L^p \exp\left(-C_1 n^{-1/2}p\right),
		\end{equation}
		for some suitable constant $C_1 > 0.$
	\end{lemma} 
	
	\begin{proof}
		We shall use the Gershgorin Circle theorem presented in Theorem~\ref{thm:Gershgorin}. Each of the eigenvalues of $\bOmega$, given by $\eig_1(\bOmega)\leq \cdots \leq \eig_p(\bOmega)$, lies in the interval $\cup_{j=1}^p\left[ \omega_{jj} \mp \sum_{k = 1, k \neq j}^{p}|\omega_{kj}|\right].$ This implies,
		$$\Pi(\bOmega \in \mathcal{M}_p^+(L)) \geq \Pi\left(\min_j(\omega_{jj} - \sum_{k = 1, k \neq j}^{p}|\omega_{kj}|) > 0, \bOmega \in \mathcal{M}_p^+(L) \right).$$ 
		For the constraint that $\min_j(\omega_{jj} - \sum_{k = 1, k \neq j}^{p}|\omega_{kj}|) > 0$, $$\eig_p(\bOmega) = \|\bOmega\|_{(2,2)} \leq \|\bOmega\|_{(1,1)} = \max_j(\omega_{jj} + \sum_{k = 1, k \neq j}^{p}|\omega_{kj}|) \leq 2\max_j \omega_{jj},$$ and,
		$$\eig_1(\bOmega) \geq \min_j(\omega_{jj} - \sum_{k = 1, k \neq j}^{p}|\omega_{kj}|).$$
		Thus,
		\begin{eqnarray}
			\label{eqn:constrained-omega-1}
			&&\Pi(\bOmega \in \mathcal{M}_p^+(L))\nonumber \\
			& \geq & \Pi(L^{-1} \leq \min_j(\omega_{jj} - \sum_{k = 1, k \neq j}^{p}|\omega_{kj}|) \leq 2 \max_j \omega_{jj} \leq L) \nonumber \\
			&\geq & \Pi(L^{-1} \leq \min_j(\omega_{jj} - L^{-1}) \leq 2\max_j \omega_{jj} \leq L \mid \max_{k \neq j}|\omega_{kj}| < (Lp)^{-1})\Pi(\max_{k \neq j}|\omega_{kj}| < (Lp)^{-1}) \nonumber \\
			&=& \Pi(L^{-1} \leq \min_j(\omega_{jj} - L^{-1}) \leq 2\max_j \omega_{jj} \leq L)\Pi(\max_{k \neq j}|\omega_{kj}| < (Lp)^{-1}).
		\end{eqnarray}
		Note that,
		\begin{eqnarray}
			\label{eqn:constrained-omega-2}
			\Pi(L^{-1} \leq \min_j(\omega_{jj} - L^{-1}) \leq 2\max_j \omega_{jj} \leq L)&\geq & \Pi(2L^{-1} \leq \omega_{jj} \leq L/2,\; 1 \leq j \leq p) \nonumber \\
			&=& \prod_{j=1}^{p}\Pi(2L^{-1} \leq \omega_{jj} \leq L/2) \sim L^p.
		\end{eqnarray}
		Also, from (\ref{eqn:lemma-prior-prob-1}) in Lemma~\ref{lemma:prior-prob}, we get,
		\begin{eqnarray}
			\label{eqn:constrained-omega-3}
			\Pi(\max_{k \neq j}|\omega_{kj}| < (Lp)^{-1}) &=& \prod_{k \neq j}\left\lbrace 1 - \Pi(|\omega_{kj}| > (Lp)^{-1}) \right\rbrace \nonumber \\
			&\geq & (1 - C_0 a^{1/2} Lp)^{p^2} \geq \exp\left( -C_1 a^{1/2} Lp^3\right) \nonumber \\
			&\geq & \exp\left(-C_1 n^{-1/2}p\right). 
		\end{eqnarray}
		The last inequality follows from the fact that $a^{1/2} < L^{-1}n^{-1/2}p^{-(2+u)},\, u > 0.$
		Therefore, combining (\ref{eqn:constrained-omega-1}), (\ref{eqn:constrained-omega-2}) and (\ref{eqn:constrained-omega-3}), we get, $\Pi(\bOmega \in \mathcal{M}_p^+(L)) \gtrsim L^p \exp\left(-C_1n^{-1/2}p\right),$ thus completing the proof.
	\end{proof}
	
	\begin{corollary}
		\label{corr:GHS-prior-constraint}
		The above lemma holds true for the graphical horseshoe prior as well under the same conditions on the global shrinkage parameter.
	\end{corollary}
	
	\begin{proof}
		The proof of this result is exactly similar to that of Lemma~\ref{lemma:prior-constraint}. The lower bound on the off-diagonal entries follows immediately from Corollary~\ref{corr:GHS-prior-prob}. The rest of the arguments remain intact.
	\end{proof}
	
	\begin{lemma}[Lemma A.3 in \citet{bickel2008regularized}]
		Let $\Zmat_i \stackrel{iid}{\sim} \mathcal{N}_p(\zerovec, \bSigma),\; \mathrm{eig}_p(\bSigma) \leq \varepsilon_0 < \infty.$ Then, if $\bSigma = (\!(\sigma_{ij})\!)$,
		$$\Pr\left[\left|\sum_{i=1}^{n}Z_{ij}Z_{ik} - \sigma_{jk} \right| \geq nt \right] \leq c_1\exp(-c_2nt^2),\, |t| \leq \delta,$$
		where $c_1,c_2$ and $\delta$ depend on $\varepsilon_0$ only.
		\label{lemma:lemma1}
	\end{lemma}
	
	\section{Proof of Corollary \ref{cor-post-conv}}\label{app-cor-post-conv}
	The proof of this result is exactly similar to that of Theorem~\ref{thm:main-post-conv}. The proof of the latter relies on Lemma~\ref{lemma:prior-prob} and Lemma~\ref{lemma:prior-constraint} that are specific to the graphical horseshoe-like prior, and the corollaries given by Corollary~\ref{corr:GHS-prior-prob} and Corollary~\ref{corr:GHS-prior-constraint} are respectively their counterparts corresponding to the graphical horseshoe prior. The utilization of the general lemma on Kullback--Leibler distance computations as outlined in Lemma~\ref{lemma:KL} remains identical in the present case.
	
	\section{Proof of Lemma \ref{lemma:concave}}\label{app-lemma-concave}
	We will prove concavity by proving the second derivative is negative. By direct calculations:
	\begin{equation}
		\frac{d^2}{dx^2}(pen_{a}(x)) = \frac{d^2}{dx^2}(-\log \log ( 1 + \frac{a}{x^2})) = - \frac{2a \left((a+3x^2)\log(1+a/x^2)-2a \right)}{x^2(a+x^2)^2(\log^2(1+a/x^2))}. \label{eq:dd2}
	\end{equation}
	Since the denominator of the RHS in \eqref{eq:dd2} is always positive, we can investigate the sign of the double derivative of the above penalty function by considering only the numerator, and furthermore as $a > 0$, we need the following to hold to prove concavity:
	\beq
	(a+3x^2)\log(1+a/x^2)-2a \ge 0. \label{eq:dd3}
	\eeq
	Substituting $\log(1+a/x^2)$ by $z$, so that $x^2 = a/(\exp(z) - 1)$, we have $z \ge 0$, and the RHS of \eqref{eq:dd3} is given by,
	\begin{align*}
		(a+3x^2)\log(1+a/x^2)-2a & = \left(a + \frac{3a}{\exp(z) - 1}\right)z - 2a \\
		& = a \left(\frac{3z + (z-2)(\exp(z) - 1)}{\exp(z) - 1}\right) \\
		& > a \left(\frac{z(1+z)}{\exp(z) - 1}\right) > 0, \quad \mathrm{since} \exp(z) > 1 + z. 
	\end{align*} 
	This proves the (strict) concavity of the graphical horseshoe-like penalty function. 
	

	\section{Estimating the Global Scale Parameter}
	\label{global_scale_param_computation}
	We use the technique of \citet{piironen2017sparsity} to tune the global parameter for GHS-LIKE-ECM via an estimate of the \emph{effective model size}.
	Consider a linear regression model $y_i = \thetavec^T \bx_i + \epsilon_i,\,\epsilon_i \sim \mathcal{N}(0,\sigma^2)$ for $i = 1,\ldots,n$, where $\thetavec = \{\theta_j\}$ and $\bx_i$ are $p$-dimensional vectors. Consider the global-local shrinkage priors of the form $\theta_j \sim \mathcal{N}(0, \lambda_j^2 a),\, \lambda_j \sim \pi(\lambda_j)$. Then, assuming the design matrix to be orthogonal, the shrinkage estimates of the elements of $\thetavec$ can be written as, $\bar{\theta}_j = (1-\kappa_j)\hat{\theta}_j$. In this context, $\kappa_j$ is called shrinkage coefficient, which takes the form $\left(1+n\sigma^{-2}\tau^2\lambda_j^2\right)^{-1}$ and $\hat{\theta}_j$ is the ordinary least squares (OLS) estimate. For the horseshoe-like prior, $\pi(\theta_j \mid u_j, a)  \sim \mathcal{N}\left(0,a/(2u_j)\right),\text{and } \pi(u_j)  = (1-\exp(-u_j))/(2\pi^{1/2}{u_j}^{3/2})$. Hence, the shrinkage coefficient is $\left(1+n\sigma^{-2} a (2u_j)^{-1}\right)^{-1}$.
	
	\citet{piironen2017sparsity} define the effective model size as $m_{\text{eff}}= \sum_{j= 1}^{p}(1-\kappa_j)$. In order to compute the global scale parameter $a$ or to decide a prior for it, they set $\mathbb{E}(m_\text{eff}) = p_0$, which is the expected number of non-zero elements in $\thetavec$. In the context of our problem, we need to find an expression for $\mathbb{E}(\kappa_j)$ in order to solve for $a$ in  $\mathbb{E}(m_\text{eff}) = p_0$. Using the standard Jacobian technique we get the density of $\kappa_j$ as,
	\begin{equation*}
		\pi(\kappa_j) = \dfrac{1}{2\pi^{1/2}}\kappa_j^{-3/2}(1-\kappa_j)^{-1/2}\left(\dfrac{na}{2\sigma^2}\right)^{-1/2}\left\{1- \exp\left(-\frac{\kappa_j}{1-\kappa_j}\frac{na}{2\sigma^2}\right)\right\}.
	\end{equation*}
	After some trivial variable transforms, $\mathbb{E}(\kappa_j \mid a)$ can be written as,
	\begin{equation*}
		\mathbb{E}(\kappa_j \mid a) =  \frac{2\sigma}{(2\pi n a)^{1/2}}\int_{0}^{\pi/2}\left\{1- \exp\left(-\frac{na}{2\sigma^2}\tan^2\eta\right)\right\}d\eta. 
	\end{equation*}
	As a closed form solution of the above integral is not available, we set $na(2\sigma^2)^{-1} = m^2$ and approximate $\exp\left(-na \tan^2\eta/(2\sigma^2)\right)$ by a polynomial of order 16 using Taylor's series about $\eta=0$. After integrating the approximated polynomial with respect to $\eta$, we get a polynomial of order 31 in $m$ to solve for the value of global scale parameter $a$. Thus, $\mathbb{E}(m_\text{eff}) = p_0$ gives,
	\begin{equation*}
		1 - \sum_{r = 0}^{31} c_r m^r = \frac{p_0}{p},
	\end{equation*}
	where $c_r,\, r = 0,\ldots,31$ are the coefficients of the polynomial which are obtained after Taylor's expansion followed by integration. In all our simulations and real data applications, we fix $p_0/p = 2/(p-1)$ and assume that $\sigma^2 =1$. For a given value of $(n,p)$, we can get a value for the global scale parameter $a = 2m^2/n$ by solving the above equation for $m$. For our simulations and real data application, we got 15 pairs of complex conjugates while solving the above equation (in \texttt{R}, using function \texttt{polyroot()}) and only one real positive value for $m$. The values of global scale parameter hence obtained for simulations and real data experiment are as follows:
	\begin{enumerate}
		\itemsep0em
		\item \textbf{Table~\ref{simulation_p_100_n_120_part_1}, \ref{simulation_p_100_n_120_part_2}}: $(n,p) = (120,100)$, estimate of the global scale parameter $a = 0.0143$.
		\item \textbf{Table~\ref{simulation_p_200_n_120_part_1}, \ref{simulation_p_200_n_120_part_2}}: $(n,p) = (120,200)$, estimate of the global scale parameter $a = 0.0169$.
		\item \textbf{Proteomics Data:} $(n,p) = (33, 67)$, estimate of the global scale parameter $a = 0.0519$.
	\end{enumerate}
	\section{Diagnostics: Choice of Starting Values for the ECM Algorithm and Trace Plots for the ECM and MCMC Algorithms}
	\label{sec:simulation_study_2}
	Since the likelihood surface under the GHS-Like prior is likely highly multi-modal, and the ECM algorithm is only guaranteed to find a local mode, we provide additional numerical results investigating the effect of starting values on the estimates. Given a true precision matrix $\bOmega_0$ and $(n,p) = (50,100)$ we generate 50 data sets, and perform estimation with 1, 10, 20 and 50 randomly chosen starting points. The accuracy measures of these estimates are represented as 1*, 10*, 20* and 50* in the Table \ref{multi_start_point_supp_table_n_50_p_100}. In general, we observe that the estimates from 50 different starting points perform the best in terms of Stein's loss, Frobenius norm, and TPR; while being slightly worse in terms of FPR and MCC. 
	
	In the presence of a highly multimodal likelihood surface, it is safe to believe that true signal which might be missed for any given starting point. Hence averaging across different starting values leads to an improvement in terms of most metrics and this what we choose to follow in our examples. Nevertheless, it is reassuring to see the final results are not too sensitive to the choice of starting values.

	\begin{table}[!h]
		\centering
		\caption{Mean (sd) Stein's loss, Frobenius norm, true positive rates and false positive rates, Matthews Correlation Coefficient of precision matrix estimates for GHS-LIKE-ECM over 50 data sets with $p=100$ and $n=50$. The best performer in each row is shown in bold. Average CPU time is in seconds.
		}
		\label{multi_start_point_supp_table_n_50_p_100}
		\begin{footnotesize}
			\noindent\makebox[\textwidth]{%
				\begin{tabular}{|l| rrrr | rrrr|}
					\hline
					& \multicolumn{4}{c|}{Random} & \multicolumn{4}{c|}{Hubs} \\
					nonzero pairs & \multicolumn{4}{c|}{35/4950} & \multicolumn{4}{c|}{90/4950} \\
					nonzero elements & \multicolumn{4}{c|}{$\sim -\mathrm{Unif}(0.2,1)$} & \multicolumn{4}{c|}{0.25} \\
					$p=100, n=50$ & 1*   & 10*   & 20*   & 50*   &  1*   & 10*   & 20*   & 50*     \\
					\hline
					
					Stein's loss & 9.624 & 8.503 & 8.196 & \textbf{8.302}
					& 12.563 & 10.494 & 10.408 & \textbf{10.268}  \\
					& (0.915) & (0.801) & (0.749) & (0.749) 
					& (0.83) & (0.885) & (0.831) & (0.81)\\
					
					F norm & 3.674 & 3.344 & 3.286 & \textbf{3.279} 
					& 4.166 & 3.672 & 3.645 & \textbf{3.616}  \\
					& (0.237) & (0.191) & (0.191) & (0.178)
					& (0.197) & (0.188) & (0.171) & (0.174) \\
					
					TPR & 0.703 & 0.816 & 0.814 & \textbf{0.825}  
					& 0.551 &  0.756 & 0.766 & \textbf{0.772}  \\		
					& (0.044) & (0.042) & (0.046) & (0.039) 
					& (0.049) & (0.053) & (0.054) & (0.053)  \\
					
					FPR & \textbf{0.021} & 0.054 & 0.058 & 0.063 
					& \textbf{0.015} & 0.038 & 0.041 & 0.044\\
					& (0.002) & (0.004) & (0.004) & (0.005) 
					& (0.002) & (0.004) & (0.004) & (0.004)  \\
					
					MCC & \textbf{0.329} & 0.271 & 0.26 & 0.253 
					& \textbf{0.464} & 0.435 & 0.426 & 0.419  \\
					& (0.024) & (0.014) & (0.013) & (0.011)  
					& (0.034) & (0.027) & (0.025) & (0.023) \\

					Avg CPU time & 6.735 & $\cdots$ & $\cdots$ & $\cdots$
					& 5.378 &$\cdots$&$\cdots$& $\cdots$ \\
					
					\hline
					
					& \multicolumn{4}{c|}{Cliques negative} & \multicolumn{4}{c|}{Cliques positive} \\
					nonzero pairs & \multicolumn{4}{c|}{30/4950} & \multicolumn{4}{c|}{30/4950} \\
					nonzero elements & \multicolumn{4}{c|}{-0.45} & \multicolumn{4}{c|}{0.75} \\
					$p=100, n=50$ & 1*   & 10*   & 20*   & 50*   &  1*   & 10*   & 20*   & 50*     \\
					\hline
					Stein's loss & 9.149 & 7.417 & 7.276 & \textbf{7.289}
					& 13.896 & 9.156 & 8.719 & \textbf{8.65}  \\
					& (0.774) & (0.673) & (0.673) & (0.67) 
					& (1.032) & (0.83) & (0.822) & (0.848)\\
					
					F norm & 3.746 & 3.231 & 3.18 & \textbf{3.174} 
					& 5.453 & 4.178 & 3.995 & \textbf{3.974}  \\
					& (0.265) & (0.263) & (0.215) & (0.255)
					& (0.245) & (0.241) & (0.268) & (0.261) \\
					
					TPR & 0.911 & 0.995 & 0.999 & \textbf{1}  
					& 0.741 &  \textbf{0.997} & \textbf{0.997} & \textbf{0.997}  \\		
					& (0.029) & (0.012) & (0.005) & (0) 
					& (0.048) & (0.001) & (0.01) & (0.01)  \\
					
					FPR & \textbf{0.021} & 0.053 & 0.058 & 0.064 
					& \textbf{0.023} & 0.051 & 0.055 & 0.059\\
					& (0.002) & (0.004) & (0.005) & (0.005) 
					& (0.002) & (0.003) & (0.004) & (0.005)  \\
					
					MCC & \textbf{0.433} & 0.312 & 0.3 & 0.287 
					& \textbf{0.344} & 0.318 & 0.308 & 0.298 \\
					& (0.016) & (0.013) & (0.012) & (0.011)  
					& (0.024) & (0.009) & (0.01) & (0.01) \\
					
					Avg CPU time & 5.08 & $\cdots$ & $\cdots$ & $\cdots$
					& 5.088 &$\cdots$&$\cdots$& $\cdots$ \\
					\hline
			\end{tabular}}
		\end{footnotesize}
	\end{table}
	
    Further, Figure~\ref{sample_trace_plot} shows a sample trace plot of log-likelihood when the precision matrix was estimated for a representative data set using GHS-LIKE-ECM and GHS-LIKE-MCMC. It is apparent that convergence to a local maximum (for ECM) and to the stationary distribution (for MCMC) occur relatively quickly. Similar behavior was observed in all other settings. 

	\begin{figure}[!htb]
	    \centering
	    \includegraphics[scale =0.4]{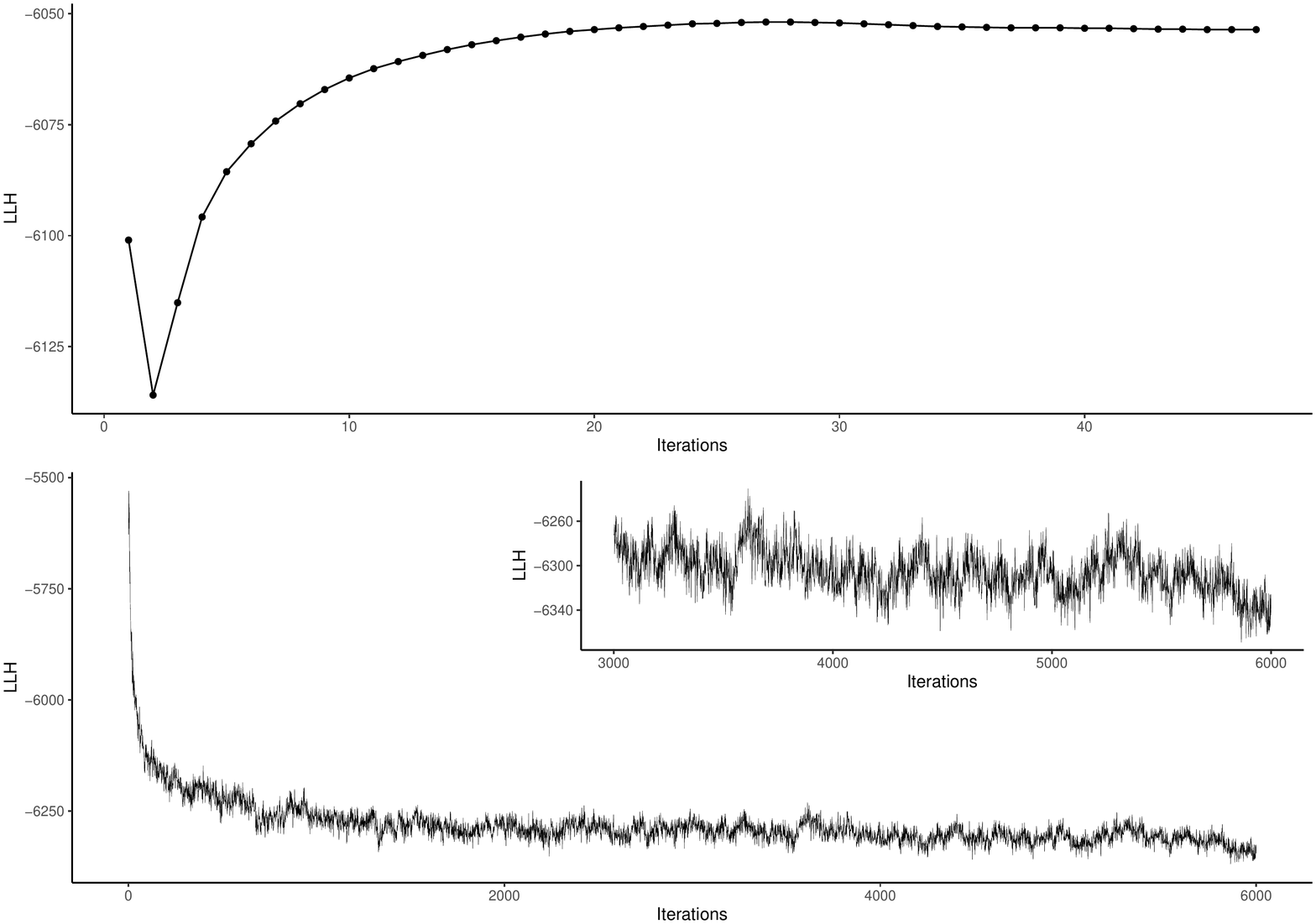}
	    \caption{Top and bottom panels show the plot of log-likelihood (LLH) vs. Iterations when the precision matrix was estimated for a representative data set using GHS-LIKE-ECM and GHS-LIKE-MCMC procedures respectively, for `Hubs' structure when $n=120,\, p=100$. The inset plot in the bottom panel shows the zoomed-in version of the plot after the burn-in period.}
	    \label{sample_trace_plot}
	\end{figure}
	\section{Additional Details on the Proteomics Data}
	\label{app-real}
	Table~\ref{nodes_num_and_protein_names} provides the map between the node numbers and protein names in Figure~\ref{graph_est_all_methods}.
	\begin{table}[!ht]
		\caption{Map between node numbers and protein names in Figure~\ref{graph_est_all_methods}.}
		\label{nodes_num_and_protein_names}
		\begin{footnotesize}
		\noindent\makebox[\textwidth]{
		\begin{tabular}{|cc|cc|cc|cc|cc|cc|cc|}
			\hline
			1  & BAK1    & 11 & MYH11                                                     & 21 & PCNA    & 31 & TP53                                                        & 41 & ATK1S1                                                  & 51 & MAPK14   & 61 & MTOR   \\ 
			2  & BAX     & 12 & \begin{tabular}[c]{@{}c@{}}RAB11A, \\ RAB11B\end{tabular} & 22 & FOXM1   & 32 & RAD50                                                       & 42 & TSC2                                                    & 52 & RPS6KA1  & 62 & RPS6   \\ 
			3  & BID     & 13 & CTNNB1                                                    & 23 & CDH1    & 33 & RAD51                                                       & 43 & INPP4B                                                  & 53 & YBX1     & 63 & RB1    \\ 
			4  & BCL2L11 & 14 & GADPH                                                     & 24 & CLDN7   & 34 & XRCC1                                                       & 44 & PTEN                                                    & 54 & EGFR     & 64 & ESR1   \\ 
			5  & CASP7   & 15 & RBM15                                                     & 25 & TP53BP1 & 35 & FN1                                                         & 45 & ARAF                                                    & 55 & ERBB2    & 65 & PGR    \\ 
			6  & BAD     & 16 & CDK1                                                      & 26 & ATM     & 36 & CDH2                                                        & 46 & JUN                                                     & 56 & ERBB3    & 66 & AR     \\ 
			7  & BCL2    & 17 & CCNB1                                                     & 27 & CHEK1   & 37 & COL6A1                                                      & 47 & RAF1                                                    & 57 & SHC1     & 67 & GATA3  \\ 
			8  & BCL2L1  & 18 & CCNE1                                                     & 28 & CHEK2   & 38 & SERPINE1                                                    & 48 & MAPK8                                                   & 58 & SRC      &    &        \\ 
			9  & BIRC2   & 19 & CCNE2                                                     & 29 & XRCC5   & 39 & \begin{tabular}[c]{@{}c@{}}ATK1, ATK2, \\ ATK3\end{tabular} & 49 & \begin{tabular}[c]{@{}c@{}}MAPK1, \\ MAPK3\end{tabular} & 59 & EIF4EBP1 &    &        \\ 
			10 & CAV1    & 20 & CDKN1B                                                    & 30 & MRE11A  & 40 & \begin{tabular}[c]{@{}c@{}}GKS3A, \\ GKS3B\end{tabular}     & 50 & MAP2K1                                                  & 60 & RPS6KB1  &    &       \\ \hline
		\end{tabular}}
\end{footnotesize}
	
	\end{table}
	
\end{document}